\documentclass[letterpaper, english, 10pt]{smfart}

\usepackage[all]{xy}

\usepackage{setspace}
\usepackage[british]{babel}

\usepackage[OT2,T1]{fontenc}

\usepackage{stmaryrd}
\usepackage{amsthm}

\usepackage{amsmath}
\usepackage[latin1]{inputenc}

\allowdisplaybreaks[1]

\usepackage{graphicx}
\usepackage{manfnt}

\usepackage{smfthm}

\usepackage{math tools} 

\usepackage{enumitem}

\usepackage{amssymb} 
\usepackage{mathrsfs} 

\usepackage[margin=3.3cm]{geometry}

\usepackage[pagebackref=true, colorlinks=true, linkcolor=black, citecolor=blue, urlcolor=black]{hyperref}

\usepackage[msc-links, lite]{amsrefs}

\DeclareSymbolFont{cyrletters}{OT2}{wncyr}{m}{n}
\DeclareMathSymbol{\Sha}{\mathalpha}{cyrletters}{"58}

\newtheorem{theoA}{Theorem}

\newtheorem*{coro*}{Corollary}
\newtheorem*{conj*}{Conjecture}
\newtheorem*{lemm*}{Lemma}

\providecommand{\smalltwomat}[4]{\left(\begin{smallmatrix}#1&#2\\#3&#4\end{smallmatrix}\right)}

\theoremstyle{definition}

\theoremstyle{remark}
\newtheorem{remark*}{Remark}

\NumberTheoremsIn{subsection}

\numberwithin{equation}{subsection}

\newcommand{\nek}{Nekov{\'a}{\v{r}}}

\newcommand{\one}{\mathbf{1}}
\newcommand{\Q}{\mathbf{Q}}

\newcommand{\GL}{\mathbf{GL}}

\newcommand{\X}{\mathscr{X}}
\newcommand{\calE}{\mathscr{E}}

\newcommand{\R}{\mathbf{R}}
\newcommand{\Z}{\mathbf{Z}}

\newcommand{\frakp}{\mathfrak{p}}

\newcommand{\into}{\hookrightarrow}

\newcommand{\Y}{\mathscr{Y}}

\newcommand{\calI}{\mathscr{I}}
\newcommand{{\calG}}{\mathscr{G}}

\newcommand{\C}{\mathbf{C}}

\newcommand{\N}{\mathbf{N}}
\newcommand{\OO}{\mathscr{O}}
\newcommand{\A}{\mathbf{A}}

\newcommand{\bks}{\backslash}
\newcommand{\baar}{\overline}

\newcommand{\eps}{\varepsilon}

\newcommand{\vpi}{\varpi}

\newcommand{\wtil}{\widetilde}

\newcommand{\ord}{\mathrm{ord}}

\newcommand{\Gal}{\mathrm{Gal}}

\newcommand{\Pic}{\mathrm{Pic}\,}

\newcommand{\Hom}{\mathrm{Hom}\,}
\newcommand{\End}{\mathrm{End}\,}

\newcommand{\llb}{\llbracket}
\newcommand{\rrb}{\rrbracket}

   \def\XXint#1#2#3{{\setbox0=\hbox{$#1{#2#3}{\int}$}
        \vcenter{\hbox{$#2#3$}}\kern-.5\wd0}}

\title[On the $p$-adic Birch and Swinnerton-Dyer conjecture over number fields]{On the $p$-adic Birch and Swinnerton-Dyer conjecture\\ for elliptic curves over number fields}
\author{Daniel Disegni}
\address{ Department of Mathematics and Statistics\\ McGill University\\
}
\email{daniel.disegni@mcgill.ca}

\begin{document}

\begin{abstract}
We formulate a multi-variable $p$-adic Birch and Swinnerton-Dyer conjecture for $p$-ordinary elliptic curves $A$ over number fields $K$. It generalises  the one-variable conjecture of Mazur--Tate--Teitelbaum, who studied the case $K=\Q$ and the  phenomenon of exceptional zeros.  We discuss old and new theoretical  evidence towards our conjecture and in particular we fully  prove it,  under mild conditions, in the following situation: $K$ is imaginary quadratic, $A=E_{K}$ is the base-change to $K$ of an elliptic curve over the rationals, and the rank of $A $ is either 0 or $1$. 

 The proof is naturally divided into a few cases. Some of them are deduced  from the purely  cyclotomic case of elliptic  curves over $\Q$, which we obtain from a refinement of recent work of Venerucci alongside the results  of Greenberg--Stevens, Perrin-Riou, and the author. 
 The only genuinely multi-variable case (rank 1, two exceptional zeros, three partial derivatives) is newly established here. 
  Its proof generalises to show  that the `almost-anticyclotomic' case of our conjecture is a consequence of  conjectures of Bertolini--Darmon on families of Heegner points,  and of (partly conjectural) $p$-adic Gross--Zagier  and Waldspurger formulas in  families.
\end{abstract}

\maketitle

\tableofcontents

\section{Introduction}
 Let $A/\Q$ be an elliptic curve, $p$ be a prime. Mazur--Tate--Teitelbaum \cite{mtt} observed long ago that if $A$  has split  multiplicative reduction at $p$, the $p$-adic $L$-function $L_{p}(A)$ of $A$ vanishes at its central argument even when the complex $L$-function $L(A,s)$ does not (due to the vanishing of the interpolation factor relating their two values). 
 They went on to   conjecture that $L_{p}(A)$ has order of vanishing   equal to exactly  one more than that of $L(A,s)$, and to give a precise conjectural formula for its leading term. 

The purpose of this work is to formulate a generalisation of their conjecture valid for elliptic curves over number fields, and to provide  evidence in its favour. A salient feature of this generalisation is the presence of several variables.

\medskip

Evidence and motivation for a conjecture are two  sides of the same mathematical coin, with the expository difference that the latter is presented before rather than after. Readers eager to look at the side of motivation may prefer to start by looking at the various formulas at the end of \S\ref{sec: bc}.

\subsection{The  conjecture}\label{sec: the conj}
Fix a rational prime~$p$ and assume throughout this paper that  $A/K$ is an elliptic curve with  ordinary (good or multiplicative) reduction at all primes $\frakp \vert p$ of $K$. The first point of difference to \cite{mtt} is that in general the existence of a $p$-adic $L$-function for $A$ is far from being known. It will therefore be a preliminary hypotheses; we formulate it after introducing some notation.
\subsubsection{$p$-adic $L$-function} Let $\Gamma $ be a  $\Z_{p}$-free quotient of the Galois group of the maximal abelian extension of $K$. Via the reciprocity law of class field theory,   
it is equipped with a surjective homomorphism 
\begin{align}\label{ell}
\ell\colon K^{\times}\bks K_{\A^{\infty}}^{\times}\to \Gamma.
\end{align}
Fix a finite extension $L$ of $\Q_{p}$ containing, for all primes $\frakp\vert p$ of good reduction, a splitting field for  the polynomial $P_{\frakp}(X)=X^{2}-a_{\frakp}X+{\rm N}(\frakp)$, where $a_{\frakp}={\rm N}(\frakp)+1-|A(k_{\frakp})|$ (here $k_{\frakp}$ is the residue field, $N(\frakp)=|k_{\frakp}|$). We let $\alpha_{\frakp}\in L$ be the unique root of $P_{\frakp}(X)$ which is a unit in $\OO_{L}$ if $\frakp $ is a prime of good reduction, and $\alpha_{\frakp}=+1$ (respectively, $\alpha_{\frakp}=-1$)  if $\frakp$ is a prime of split (respectively, non-split) multiplicative reduction. Finally, we let $\Q(\alpha):=\Q((\alpha_{\frakp})_{\frakp\vert p})\subset L$. 

The (hypothetic) $p$-adic $L$-function $L_{p}^{(\Gamma)}(A)$ will be an element of $\OO_{L}\llb \Gamma\rrb\otimes L$, which we also view as the ring $\OO(\Y_{\Gamma})^{\rm b}$ of bounded regular functions on the Cartier dual of $\Gamma$ over $L$: this  is a rigid space  $\Y_{\Gamma}/L$ whose $R$-valued points, for any affinoid $L$-algebra $R$, parametrise continuous characters $\chi\colon \Gamma\to R^{\times}$; accordingly, we will write $L_{p}^{(\Gamma)}(A, \chi)=\chi(L_{p}^{(\Gamma)}(A))\in L(\chi)$ for a character $\chi\in \Y_{\Gamma}(L(\chi))$.   

\begin{enonce*}{Hypothesis}[$L_{p}$]\label{hyp} The complex $L$-function of $A$ and its twists   $L(A, \chi, s)$ by finite characters of $\Gamma$ have analytic continuation to the entire complex plane, and there exists   an element
$$L_{p}^{(\Gamma)}(A)\in \OO_{L}\llb \Gamma\rrb\otimes L$$
satsifying the following property: for each finite extension $\Q(\alpha,\chi)$ of $\Q(\alpha)$, each finite order character $\chi\colon \Gamma\to \Q(\alpha, \chi)^{\times}$, and each pair $(\iota_{\infty},\iota_{p})$ of an embedding  $\iota_{\infty} \colon \Q(\alpha,\chi)\into \C$ and a $\Q(\alpha)$-embedding $\iota_{p}\colon \Q(\alpha, \chi)\into\baar{L}$, we have
$$\iota_{\infty} \iota_{p}^{-1} L_{p}^{(\Gamma)}(A)(\chi)= \prod_{\frakp \vert p}\iota_{\infty}e_{\frakp}(\chi_{\frakp}) \cdot {L(A, \iota_{\infty} \chi, 1)\over |D_{K}|^{-1/2} \Omega_{A}}\qquad \text{in } \iota_{\infty} \Q(\alpha, \chi),$$
where
 $\Omega_{A}$ is the N\'eron period appearing in the Birch and Swinnerton-Dyer conjecture for $A$ \cite{tate-bsd}, $D_{K}$ is the discrimimant of $K$,
    and the local factors $e_{\frakp}(\chi_{\frakp})$ are given as follows:
 \begin{equation*}
e_{\frakp}( \chi_{\frakp})=
\begin{cases}\displaystyle
 {( 1-\alpha_{\frakp} \chi(\frakp))(1-\alpha_{\frakp}'\chi(\frakp)^{-1})} &\text{\ if $\chi_{\frakp}$ is unramified,}\\
\alpha_{\frakp}^{-\mathfrak{f}_{\frakp}} \tau(\chi_{\frakp}) &\text{\ if $\chi_{\frakp}$ is ramified of conductor $\mathfrak{f}_{\frakp}$.}
\end{cases}
\end{equation*}
Here  $\alpha_{\frakp}':=\alpha_{\frakp}^{-1}$ if $E$ has good reduction and $\alpha_{\frakp}':=0$ if $E$ has multiplicative reduction, and  $\tau(\chi_{\frakp})$ is the Gau\ss\ sum $$\tau(\chi_{\frakp}, \psi_{\frakp}):=\sum_{x\in \OO_{K}/\frakp^{\mathfrak{f}_{\frakp}}}\chi_{\frakp}(x)\psi_{\frakp}(x),$$ for some choice\footnote{The function $L_{p}^{(\Gamma)}(A)=L_{p}^{(\Gamma)}(A, \psi_{p})$  then has a mild dependence  on the choice of $\psi_{p}=(\psi_{\frakp})_{\frakp\vert p}$: if $a\in \OO_{K, p}^{\times}$ and $a.\psi_{p}=(a_{\frakp}.\psi_{\frakp})$ with $a_{\frakp}.\psi_{\frakp}(x):= \psi_{\frakp}(ax)$, then  $L_{p}^{(\Gamma)}(A, a.\psi_{p})(\chi)=\chi_{p}(a)L_{p}^{(\Gamma)}(A, \psi_{p})(\chi)$. For a careful discussion of this point and formalisation of the `space of additive characters of level~$0$', see \cite[Theorem A]{dd}. In this paper we will be interested in the leading term at $\chi=\one$, which is independent of choices.}
 of  an additive character $\psi_{\frakp}$ of $ K_{\frakp}$ of level $0$.
\end{enonce*}

Note that the interpolation property determines $L_{p}^{(\Gamma)}(A)$ uniquely if it exists. When $\Gamma$ is the Galois group $\Gamma_{K}$ of the maximal $\Z_{p}$-extension of $K$ (or when $\Gamma$  is understood from context), it will be omitted from the notation.

\begin{rema} If $A/K$ is a ($p$-ordinary) elliptic curve over a number field, a generalisation of the Shimura--Taniyama--Weil conjecture predicts the existence of a ($p$-ordinary, cohomological) automorphic representation $\pi$ for $\GL_{2}/K$  such that $L(s,\pi)=L(A,s+1/2)$.
 A construction of $p$-adic $L$-function as in Hypothesis {$(L_{p})$} attached to such ordinary (more generally non-critical) representations $\pi$  has been announced by D. Hansen.
\end{rema}

Let $\mathscr{I}_{\Gamma}\subset\OO_{L}\llb \Gamma\rrb_{L}$ be the augmentation ideal (which is also the ideal of functions on $\Y_{\Gamma}$ vanishing at the trivial character $\chi=\one$), and define the \emph{order of vanishing}  ${\rm ord}_{\chi=\one}L_{p}^{(\Gamma)}(A)$ to be the largest integer $\wtil{r}\geq 0$ such that $L_{p}^{\Gamma}(A)\in \mathscr{I}_{\Gamma}^{\wtil{r}}$. Note that if $\Gamma\to \Gamma'$ is a quotient, then 
\begin{align}\label{ineq}
{\rm ord}_{\chi=\one}L_{p}^{(\Gamma)}(A)\leq {\rm ord}_{\chi=\one}L_{p}^{(\Gamma')}(A).
\end{align}
For any  $\wtil{r}\geq {\rm ord}_{\chi=\one}L_{p}(A)$, we define 
$${\rm d}^{\wtil{r}} L_{p}(A, \one)\quad \in {\rm Sym}^{\wtil{r}}\Gamma\otimes {L},$$
as the image of  $ L_{p}(A)$ in $\calI_{\Gamma}^{\wtil{r}}/\calI_{\Gamma}^{\wtil{r}+1}\cong  {\rm Sym}^{\wtil{r}}\Gamma\otimes{L}$, the last isomorphism being given by $\prod_{i=1}^{\wtil{r}}(\gamma_{i}-1)\mapsto [\gamma_{1}\otimes\cdots\otimes \gamma_{\wtil{r}}]$. When   $\wtil{r}= {\rm ord}_{\chi=\one}L_{p}(A)$, it can be thought of as the $\Gamma$-leading term of $L_{p}(A)$ at $\chi=\one$.\footnote{We should remark, to clear any possible confusion which might be genrated by our notation, that ${\rm d}^{\wtil{r}} L_{p}(A, \one)$  is the analogue of the usual Taylor coefficient  ${1\over\wtil{r}!}L^{(\wtil{r})}(A, 1)$ (and \emph{not} of the $r^{\rm th }$ derivative $L^{(\wtil{r})}(A, 1)$).}

\subsubsection{Arithmetic side: the extended Mordell--Weil and Selmer groups} The Birch and Swinnerton-Dyer conjecture relates  ${\rm ord}_{s=1}L(A,s)$ to the rank 
 $$r:={\rm rk}\, A(K).$$ For our modified case,  let   $S_{p}^{\rm exc}=S_{p}^{\rm exc}(A)$ be the set of places above $p$ over which $A$ has split multiplicative reduction, and let $r^{\rm exc}=|S_{p}^{\rm exc}|$. 
We conjecture that, if Hypothesis {$(L_{p})$} is satisfied for $(A, \Gamma)$ and  if the natural surjection $\Gamma_{K}\to \Gamma_{\Q}$ factors through $\Gamma$,\footnote{If this last condition is not satisfied, there are examples in which the order of vanishing can be higher; in fact we may even have $L_{p}^{(\Gamma)}(A)=0$ identically, see Proposition \ref{Lp for A} below.}
\begin{align}\label{ordvan}
{\rm ord}_{\chi=\one}L_{p}^{(\Gamma)}(A)=\wtil{r}:=r+r^{\rm exc}.
\end{align}

The integer $\wtil{r}$ is interpreted as the rank of the \emph{extended Mordell--Weil group} $A^{\dagger}(K)$. In order to define it, recall first  that Tate proved that  if $A'/F$ is an elliptic curve with split multiplicative reduction over a  non-archimedean local field $F$, there is a rigid analytic  isomorphism ${\bf G}_{m, F}^{\rm an}/ q^{\Z}\to A'^{\rm an}$ for some
  $q\in F^{\times}$ (called the \emph{Tate period} of $A'/F$ and chosen to satisfy ${\ord}_{\frakp}(q)>0$). Then, letting $q_{A, \frakp}$ be the Tate period of $A/{K_{\frakp}}$ for $\frakp \in S_{p}^{\rm exc}$,
\begin{align*}
A^{\dagger}(K):=A(K) \oplus \bigoplus_{\frakp \in  S_{p}^{\rm exc}}\Z\, q_{A, \frakp},
\end{align*}
 and $\wtil{r}={\rm rk}\, A^{\dagger}(K)$.
 
The work of \nek \ (see \cite[\S7.14]{nekheights}, \cite[\S9.6.7]{nek-selmer}) introduces an  extended Selmer group $\wtil{H}^{1}_{f}(K, V_{p}A)$,\footnote{In the case at hand this  was also defined \emph{ad hoc} in \cite{mtt}.}, a $\Q_{p}$-vector space containing $A^{\dagger}(K)\otimes \Q_{p}$ and   explicitly described as
  \begin{align}\label{ext-sel}
\wtil{H}^{1}_{f}(K, V_{p}A)\cong H^{1}_{f}(K, V_{p}A) \oplus \bigoplus_{\frakp \in  S_{p}^{\rm exc}}\Q_{p} q_{A, \frakp}, 
\end{align}
where $H^{1}_{f}(K, V_{p}A)$ is  the (Bloch--Kato) Selmer group of $A$. We have  
$A^{\dagger}(K)_{\Q_{p}}=  \wtil{H}^{1}_{f}(K, V_{p}A)$, {provided that} $\Sha(A)[p^{\infty }]$ is finite. Our conjectures will implicitly assume this to be the case, in the sense that they would likely need to be modified if $\Sha(A)[p^{\infty }]$ were to be infinite.\footnote{Cf. the remark \emph{On the rational part in the arithmetic side} in \S\ref{misc}.}

\subsubsection{Arithmetic side: the regulator} We now turn to describing the various arithmetic  ingredients which will combine into the conjectural value of the leading term. For the complex $L$-function, the leading term is conjecturally given by the product of a rational number and of the  N\'eron--Tate regulator on $A(\Q)$.  In our case, the N\'eron--Tate height pairing
\begin{align}\label{NT}
h_{\rm NT}\colon A(K)\times A(K)\to \R,
\end{align}
whose \emph{discriminant} (see Definition \ref{def-reg} below: it also accounts for $|A(K)_{\rm tors}|^{2}$)
 on $A(K)$ is denoted by $R_{\rm NT}(A)$,
 is replaced by the extended height pairing 
 \begin{align}\label{ext-ht}
\wtil{h}_{\ell}\colon \wtil{H}^{1}_{f}(K, V_{p}A)\times \wtil{H}^{1}_{f}(K, V_{p}A)\to \Gamma\otimes L;
\end{align}
defined by \nek\ (\emph{locc. citt.}). (More precisely, \nek\ defines a pairing $\wtil{h}$ with values in $K^{\times}\bks K_{\A^{\infty}}^{\times}\hat{\otimes} L=\Gamma_{K}\hat{\otimes} L$. The pairing $\wtil{h}_{\ell}$ is its image under \eqref{ell}.)
 The regulator term 
$$\wtil{R}_{\ell}(A)\in {\rm Sym}^{\wtil{r}}\Gamma\otimes L$$
is defined as  the discriminant   of \eqref{ext-ht} on  $A^{\dagger}(K)$.
A concrete description of the height pairing \eqref{ext-ht} will be given in \S\ref{phts}. For now we will just say that 
(a) it extends the ``canonical'' height pairing 
\begin{align}\label{can-ht}
h_{\ell}=h_{\ell}^{\rm can}\colon A(K)\times A(K)\to \Gamma\otimes L
\end{align}
also defined in \cite[\S 7.14]{nekheights};  (b) in the case $K=\Q$ considered in \cite{mtt}, our regulator $\wtil{R}_{\ell}(A)$ differs from the regulator of \cite{mtt}  by a factor ${\rm ord}_{p}(q_{A,p})$ if $S_{p}^{\rm exc}=\{p\}$  (whereas the two  coincide  if $S_{p}^{\rm exc}=\emptyset$).

\subsubsection{The conjecture}
We first recall the classical Birch and Swinnerton-Dyer (${\rm BSD}_{\infty}$) conjecture for  elliptic curves over  number fields. We present it as a ``hypothesis'', as we  will prefer to formulate its $p$-adic analogue based on the hypothesis that the first two predictions of (${\rm BSD}_{\infty}$) hold.
\begin{enonce*}{Hypothesis}[${\textup{BSD}}_{\infty}$]\label{BSDc} Let $A$ be an elliptic curve over a number field $K$.
\begin{enumerate}
\item\label{BSD1} The integer $r_{\rm an}(A):={\rm ord}_{s=1}L(A,s)$ equals $r(A):={\rm rk}\, A(K)$;
\item\label{BSD2} Letting $c_{v}(A)$ denote the local Tamagawa number of $A$ at a prime $v$ of $K$ and $R_{\rm NT}(A)$  be the regulator of \eqref{NT} on $A(K)$,  the number 
$$ |\Sha(A)|_{\rm an}:={L^{(r)}(A, 1)\over r! |D_{K}|^{-1/2} R_{\rm NT}(A) \Omega_{A}\prod_{v}c_{v}(A)}$$
belongs to $\Q^{\times}$.
\item\label{BSD3} The Tate--Shafarevich group $\Sha(A)$ is finite, and its order equals $|\Sha(A)|_{\rm an}$. 
\end{enumerate}
\end{enonce*}

We may now state the $p$-adic Birch and Swinnerton--Dyer conjecture (${\rm BSD}_{p}$).
\begin{conj*}[${\textup{BSD}}_{p}$] Let $A$ be an elliptic curve over a number field $K$, $p$ a rational prime such that $A$ has ordinary reduction at all the primes $\frakp\vert p$ of $K$. Let $\Gamma$ be a $\Z_{p}$-free quotient of $\Gal(K^{\rm ab}/L)$.
 Suppose that Hypotheses {$(L_{p})$} and $({\rm BSD}_{\infty})$--\ref{BSD1}-\ref{BSD2} are satisfied and that ${\rm ord}_{s=1}L(A,s)=r$. 
 Let $S_{p}$ be the set of primes  of $K$ above $p$, $S_{p}^{\rm exc}\subset S_{p}$ its subset of primes of split multiplicative reduction for $A$, and let
$$\wtil{r}:= r + |S_{p}^{\rm exc}|.$$ 

Let $ \wtil{e}_{\frakp}(\one):=e_{\frakp}(\one) $ if $\frakp\in S_{p}-S_{p}^{\rm exc}$, and $ \wtil{e}_{\frakp}(\one):=  {\rm ord}_{\frakp}(q_{A, \frakp})^{-1}$ if $\frakp\in S_{p}^{\rm exc}$.

Then $L_{p}^{(\Gamma)}(A)$ vanishes at $\chi=\one$ to order at least $\wtil{r}$, we have\footnote{This isomorphism is equivalent to the finiteness of the $p$-part $\Sha(A)[p^{\infty}]$ of the Tate--Shafarevich group of $A$.} $A^{\dag}(K)_{\Q_{p}}\cong \wtil{H}^{1}_{f}(K, V_{p}A) $ and their dimension is $\wtil{r}$, 
and
\begin{align}\label{mainid}
{\rm d}^{\wtil{r}}L_{p}^{(\Gamma)}(A,\one)=
\prod_{\frakp\vert p} \wtil{e}_{\frakp}(\one)
\cdot
\wtil{R}_{\ell}(A)
\cdot
|\Sha(A)|_{\rm an}
\prod_{v}c_{v}(A) 
 \quad \text{in\ } {\rm Sym}^{\wtil{r}}\Gamma\otimes L.
\end{align}
\end{conj*}
\begin{rema} If the conjecture holds, the order of vanishing at $\one$ of $L_{p}^{(\Gamma)}(A)$ is $\wtil{r}$ if and only if  $\wtil{R}_{\ell}(A)\neq 0$, i.e. if the extended height pairing \eqref{ext-ht} is non-degenerate. This is conjectured to hold under the assumptions stated before the conjectured equality \eqref{ordvan} (which is thus recovered); it can fail in general.
\end{rema}
\begin{rema} The rank of $\Gamma_{K}$ is $1+s+{\delta}$, where $s$ is the number  of complex places of $K$ and $\delta=\delta_{K,p}$ is the Leopoldt defect of $K$ at $p$ (conjectured, and known if $K$ is abelian, to be $0$). Therefore $${\rm rank}\,{\rm Sym}^{\wtil r}\Gamma_{K} = {s+\delta+\wtil r  \choose \wtil r}.$$
\end{rema}
\begin{rema}It is easy to see that if the conjecture holds for $\Gamma$, then it holds for any quotient $\Gamma'$ of $\Gamma$. One further important case of $({\textup{BSD}}_{p})$  was considered before, namely when $A=E_{K}$ is the base-change of an elliptic curve over $\Q$ to an imaginary quadratic field $K$ and $\Gamma=\Gamma^{-}$ is the rank-$1$ quotient of $\Gamma_{K}$ on which the complex conjugation $c\in \Gal(K/\Q)$ acts by $-1$. This \emph{anticyclotomic} single-variable conjecture was stated and studied by Bertolini--Darmon in a series of works beginning with \cite{bdMT}. 
\end{rema}

\subsection{Evidence} The rest of this paper presents the evidence for Conjecture {$({\textup{BSD}}_{p})$}, whose compatibility with the conjecture of \cite{mtt} is recalled in Proposition \ref{compatibility}.  

The results are of course concentrated in the cases where something is known for the classical conjecture  {$({\textup{BSD}}_{\infty})$}, namely when $A$ is an elliptic curve over $\Q$ (or a totally real field $F$) or its base-change to an imaginary quadratic field (or a quadratic CM extension of $F$), and the archimedean analytic rank is at most~$1$. The cases of higher-degree totally real and CM fields are largely analogous to the cases of $\Q$ and imaginary quadratic fields, but the results there   a little more fragmentary and a little less clean.  We will therefore limit our discussion of  them to various remarks (\ref{rem dd}, \ref{rem spiess}, \ref{rem hung})  throughout the main body of the text. 

\begin{theoA}\label{mainQ} Let $E/\Q$ be an elliptic curve of conductor $N$ with ordinary reduction at the prime~$p$.  
 Suppose that  $r_{\rm an}:={\rm ord}_{s=1} L(E,s)\leq 1$ 
and  that 
\begin{enumerate}
\item[{\rm ($*$)}]\label{exist mult} if $r_{\rm an}=1$ and $S_{p}^{\rm exc}(E)\neq \emptyset$, then $p\geq 5$ and there exists another prime $m\neq p$ of multiplicative reduction for $E$.
\end{enumerate}
 Then Hypotheses {$(L_{p})$} and $({\rm BSD}_{\infty})$--\ref{BSD1}-\ref{BSD2} are satisfied, and  Conjecture {$({\textup{BSD}}_{p})$} holds.
 \end{theoA}

\begin{theoA}\label{mainK} Let $E/\Q$ be an elliptic curve of conductor $N$ with ordinary reduction at the prime~$p$. Let $K$ be  an imaginary quadratic field of discriminant prime to $Np$ and let  $A=E_{K}$ be the base-change. 
 Suppose that  $r_{\rm an}:={\rm ord}_{s=1} L(A,s)\leq 1$ 
and that
\begin{enumerate}
\item[{\rm ($*'$)}] if $r_{\rm an}=1$ and $S_{p}^{\rm exc}(A)\neq \emptyset$, then $p\geq 5$ and there exists another prime $m\neq p$ of multiplicative reduction for $E$, and moreover:
\begin{enumerate}
\item if $p$ is inert in $K$, then the prime $m$ can be chosen to be also inert in $K$;
\item if $p$ splits in $K$, then every $v\vert N$ splits in $K$.
\end{enumerate}
\end{enumerate}
  Then Hypotheses {$(L_{p})$} and $({\rm BSD}_{\infty})$--\ref{BSD1}-\ref{BSD2} are satisfied, and  Conjecture {$({\textup{BSD}}_{p})$} holds.

\end{theoA}

\subsubsection{A new exceptional zero height formula} We highlight the main new case of Theorem \ref{mainK}, thereby also exemplifying the type of formulas predicted by our conjecture.  Let us introduce some notation. Given $A$, $\ell$, $\frakp\in S_{p}^{\rm exc}$ as above, the \emph{$\mathcal{L}$-invariant} (as introduced in \cite{mtt}) is
\begin{align}\label{Linv}
\mathcal{L}_{\frakp, \ell}(A):= {\ell_{\frakp}(q_{A,\frakp})\over {\rm ord}_{\frakp}(q_{A,\frakp})}\in \Gamma\otimes\Q_{p}.
\end{align}
We place ourselves in the setup of Theorem \ref{mainK}.
The dihedral action of  $\Gal(K/\Q)$ induces an eigenspace decomposition  $\Gamma_{K}=\Gamma^{+}\times \Gamma^{-}$, with $\Gamma^{+}$ identified with $\Gamma_{\Q}$ via the ad\`elic norm and each of the factors free of rank $1$ over $\Z_{p}$; we obtain a corresponding decomposition $\Y=\Y^{+}\times \Y^{-}$. We denote by 
$$({\rm d}^{-})^{i}({\rm d}^{+})^{j}\colon \calI_{\Gamma_{K}}^{i+j}\to ({\rm Sym}^{i}\Gamma^{-}\otimes {\rm Sym}^{j}\Gamma^{+})\otimes L$$
the projection, and by 
$$\mathcal{L}_{\frakp}^{\pm}(A):=\mathcal{L}_{\frakp, \ell^{\pm}}(A),$$
where $\ell^{\pm}\colon \Gamma \to \Gamma^{\pm}$ is the projection. Finally, there is another pairing $h^{\rm norm}_{\ell}$ on $H^{1}_{f}(K, V_{p}A)$, which is  defined  by means of universal norms and whose relation with $h^{\rm can}$ will be recalled in 
\eqref{h norm} below: we let
$$R^{\rm norm,+}(A)$$ 
be
the discriminant of $h^{\rm norm, +}=h^{\rm norm}_{\ell^{+}}$  on $A(K)$.
\begin{theoA}\label{exc} Let $E/\Q$ be an elliptic curve with conductor $N$ and  split multiplicative reduction at the prime $p\geq 5$. Let $K$ be an imaginary quadratic field such that all  primes dividing $N$ split in $K$ and let $A:=E_{K}$; then $S_{p}^{\rm exc}(A)=S_{p}=\{\frakp, \frakp^{*}\}$.
Suppose that ${\rm ord}_{s=1}L(A,s)=1$. We have  ${\rm d}^{i}L_{p}(A_{})=0$  for all $i\leq 2$, and 
\begin{gather*} 
({\rm d}^{-})^{2}{\rm d}^{+}L_{p}(A, \one)=\mathcal{L}_{\frakp}^{-}(A)\mathcal{L}_{\frakp^{*}}^{-}(A)\cdot {R}^{\rm norm,+}(A)\cdot 
|\Sha(A)|_{\rm an}
\prod_{v}c_{v}(A) 
\end{gather*}	
in $(\Gamma^{-})^{ \otimes 2}\otimes \Gamma^{+}\otimes{\Q_{p}}$.
\end{theoA}

\subsubsection{Outline of proofs} 
The proof of Theorem \ref{mainQ} in \S\ref{evidenceQ}  is reduced, according to the value of $r_{\rm an}$ and the reduction type of $E$ at $p$,   to  various formulas of Perrin-Riou \cite{PR}, Greenberg--Stevens \cite{GS}, Venerucci \cite{venerucci},  and the author \cite{dd}. The present work claims no originality for it (save perhaps for the precise determination of constants in Venerucci's formula for the case of  $r=1$ and split multiplicative reduction).

Theorem \ref{mainK} is also  separated into several cases, according as above to the value of $r_{\rm an}$, the reduction type of $E$ at $p$, \emph{and} the behaviour of $p$ in $K$.  Moreover each case is broken down into several formulas according to the natural rank one quotients ${\rm Sym}^{\wtil{r}}\Gamma_{K}\to {\rm Sym}^{i}\Gamma^{+}\otimes {\rm Sym}^{\wtil r -i}\Gamma^{-}$. (The results in each individual case    sometimes hold under   weaker assumptions than  those of Theorem \ref{mainK}.)

 We single out two cases: the  \emph{purely cyclotomic case} of $i=\wtil{r}$, which reduces to Theorem \ref{mainQ} for $E$ and the twist $E^{(K)}$, and the \emph{almost-anticylotomic case} of $i=0$ or $i=1$ (depending on the sign $\wtil \eps $ of the functional equation for $L_{p}(A)$). Our range of $\wtil r$ is low enough that those two cases, together with some identities $0=0$ deduced from parity considerations, will suffice to cover Theorem \ref{mainK}. 

In \S\ref{sec: anti}, we first review the construction of  $\Y^{-}$-families of Heegner points on $A$ in the case of $\wtil \eps =-1$, as well as analogous functions on $\Y^{-}$ if $\wtil \eps=+1$ (\emph{theta elements}). Then we (re)formulate conjectures of Bertolini--Darmon \cite{bdMT} on their leading terms at $\one\in \Y^{-}$, and describe the evidence: they   are  known under mild conditions if the rank of $A$ is at most $1$. Several cases are  due to  Bertolini--Darmon themselves \cites{BDrigid , bdCD,bdperiods}; we deduce one further case from a recent formula independently proven by Castella \cite{castella} and Molina Blanco \cite{MB}.  Finally, we state (partly conjectural) explicit versions of the $p$-adic Gross--Zagier and Waldspurger formula in anticyclotomic families, based on recent works of Chida--Hsieh \cite{CH} and the author \cite{dd}. 

The key observation made in \S\ref{evidence-imag} is that, granted the anticylotomic versions of  the $p$-adic Gross--Zagier and Waldspurger formulas, the almost-anticyclotomic case of $({\textup{BSD}}_{p})$ for $A$ is essentially equivalent to the Bertolni--Darmon conjecture for $A$. This allows to complete the proof of Theorem \ref{mainK}, of which Theorem \ref{exc} is shown to be a special case.

\subsection{Miscellaneous  remarks}\label{rmks} \label{misc} We conclude this introduction with some comments on Conjecture {$({\textup{BSD}}_{p})$}. 
\subsubsection{Supersingular primes} 
The presence of primes $\frakp \vert p$ of good supersingular reduction could also be allowed;   the associated variations, however,  would be orthogonal to the  situation of exceptional zeros, which is the most interesting phenomenon investigated in this paper. We therefore prefer to leave the extension to the interested reader, who may consult e.g. \cite{bpr} for the case of elliptic curves over $\Q$, \cite{PRbei} for the most general case (but restricted to the cyclotomic variable), and \cite{sprung} and references therein for alternative formulations. 
\subsubsection{Generalisations} It is easy to imagine that the conjecture can be extended, with appropriate variations, to the case of abelian varieties, or to twisted cases, or to modular forms of higher weight. In fact, we regard it  as a special case of a very general multi-variable conjecture on the leading terms of $p$-adic $L$-functions for deformations of (symplectic) motives, to which we hope to return  in future work.  
\subsubsection{Invitation to numerical investigations} In this paper, we have strived not for generality but for keeping the conjecture as concrete and elementary as possible: we hope this will encourage some of our readers to try and test it numerically in new cases.

\subsubsection{On the rational part in the arithmetic side} The formulation chosen for Conjecture {$({\textup{BSD}}_{p})$} is slightly spurious in that the term $|\Sha(A)|_{\rm an} $ appearing in its right-hand side should ideally be replaced by its conjectural value $|\Sha(A)|$. The reason behind  our choice is that, in practice, the study of this rational term and of Conjecture {$({\textup{BSD}}_{p})$} in our formulation are often best approached separately. Nevertheless the ``purer'' version of the conjecture would probably have the advantage of  admitting a rephrasing, \emph{up to $p$-units}, solely in terms of invariants of the \nek --Selmer complex underlying the extended Selmer group\footnote{This last formulation would not, even implicitly, depend on the finiteness of the Tate--Shafarevich group $\Sha(A)$.} (concretely, in terms  of the discriminant  of $\wtil{h}$ on $\wtil{H}^{1}_{f}(K, T_{p}A)$ and of the sizes of other cohomology groups $\wtil{H}^{i}_{f}(K, T_{p}A)_{\rm tors}$, cf. \cite[\S 0.16.2]{nek-selmer}).
\subsubsection{Relation to  Iwasawa main conjectures}
We expect the Birch and Swinnerton-Dyer conjecture for the extended Selmer group alluded to above  to be a consequence of a suitable version of the Iwasawa main conjecture and the hypothetic  non-degeneracy of $\wtil{h}$, via the formulas of \emph{loc. cit.}. In the  situation of Theorem \ref{mainK}, the relevant main conjecture is the two-variable main conjecture for $A=E_{K}$, which is proved in \cite{SU} when the root number of $A$ is $+1$. It is also to be expected that Theorem \ref{exc} up to $p$-units should follow from the Birch and Swinnerton-Dyer conjecture in anticyclotomic family for $A$; in the complementary case of good  reduction, this conjecture was recently proved by X. Wan \cite{xin-howard}

\subsubsection{On the work of Venerucci} Theorem \ref{mainQ}, in the case of exceptional zeros and rank one,  refines the recent result of Venerucci \cite[Theorem D]{venerucci} on the original conjecture of \cite{mtt}. In the same work he obtains a  a two-variable exceptional zero formula \cite[Theorem C]{venerucci}, which can be regarded as an instance of the general multi-variable conjecture alluded to above, and which was for us a source of inspiration alongside the works of Bertolini--Darmon and Castella.
\subsubsection{Exceptional $p$-adic heights formulas}
 The formula of Theorem \ref{exc} stands alongside Venerucci's \cite[Theorem D]{venerucci} (see also Proposition \ref{venD} below)
  as a second example of exceptional zero formulas for  $p$-adic heights of points on rank one curves.  Its proof only uses Heegner points through the author's formula of \cite[Theorem C]{dd} (also valid over totally real fields) and the work of Castella \cite{castella} (for $\Q$, but in principle generalisable) and Molina Blanco \cite{MB} (also valid for totally real fields, with an a priori  different notion of $\mathcal{L}$-invariants). Therefore 
   Theorem \ref{exc} should also be generalisable to totally real fields.

On the other hand, the proof of \cite[Theorem D]{venerucci}  makes essential use of Beilinson--Kato classes, whose analogue in the case of totally real fields has yet to be constructed (cf. \cite{plectic}).

\subsection{Acknowledgements} 
 I would like to thank Lennart Gehrmann and Eric Urban for conversations, and Francesc Castella and Shinichi Kobayashi for correspondence on their respective works  \cite{castella} and \cite{kob-mtt}.

\section{Foundations}\label{explicit}

We start this section with  a brief review of the various $p$-adic height pairings we will need in this paper.  Then we give some explicit versions of the  Waldspurger, Gross--Zagier, and $p$-adic Gross--Zagier formulas. 
in the first two cases these are directly taken from recent works  of Cai--Shu--Tian \cite{cst} and Chida--Hsieh \cite{CH} based on the general formulas of Waldspurger and  Yuan--Zhang--Zhang \cite{yzz}; in the latter $p$-adic case we apply the results of \cite{cst, CH}  to deduce them from  \cite{dd}. 

\subsection{$p$-adic height pairings}\label{phts} Let $A/K$ be an elliptic curve  over a number field, with ordinary reduction at all $\frakp \in S_{p}$. We will consider three height pairings associated to a ``$p$-adic logarithm'' $\ell $ as in \eqref{ell}. (The material of this subsection is entirely taken from \cite{mtt} and \cite{nekheights},  to which we refer for more details; our notation and conventions are fixed as in \cite[especially \S7]{nekheights}.) They  are analogous to  the classical N\'eron--Tate height pairing 
$$h_{\rm NT}\colon A(K)\times A(K)\to \R$$
(which, for an elliptic curve over $K$, we always take to be normalised over $K$).
The first $p$-adic pairing is the ``canonical'' height pairing $h_{\ell}=h^{\rm can}_{\ell}$ on the Bloch--Kato Selmer group $H^{1}_{f}(K, V_{p}A)$, valued in  $\Gamma \otimes L$ where $\Gamma$  and $L$ are introduced before Hyopthesis $(L_{p})$ in \S\ref{sec: the conj}. We will not recall its definition.
The second one is the extended  pairing $\wtil{h}_{\ell}$ of \eqref{ext-ht} on  $\wtil{H}^{1}_{f}(K, V_{p}A)\cong H^{1}_{f}(K, V_{p}A) \oplus \bigoplus_{\frakp \in  S_{p}^{\rm exc}}\Q_{p} q_{A, \frakp}$. It extends $h_{\ell}$, and a concrete description is  given in \cite[\S 7.14]{nekheights}; let us recall it. 
   For  $\frakp\in S_{p}^{\rm exc}$, let 
 $$\log_{A, \frakp,\ell}\colon A(K_{\frakp})  \otimes \Q_{p} \cong H^{1}_{f}(K_{\frakp}, V_{p}A) \to \Gamma\otimes\Q_{p}$$
 be the map induced from $\ell_{\frakp}|_{\OO_{K, \frakp}^{\times}}$ via $\OO_{K, \frakp}^{\times}\hat{\otimes}\Q_{p}\cong K_{\frakp}^{\times}/q_{A, \frakp}^{\Z}\hat{\otimes}\Q_{p}\cong A(K_{\frakp})\otimes \Q_{p}$. Then 
 $\wtil{h}_{\ell}$  is the symmetric bilinear pairing on $\wtil{H}^{1}_{f}(K, V_{p}A)\cong H^{1}_{f}(K, V_{p}A) \oplus \bigoplus_{\frakp \in  S_{p}^{\rm exc}}\Q_{p} q_{A, \frakp}$
 given by 
 \begin{align*}
 \wtil{h}_{\ell}(x,y) & =   h_{\ell}(x,y)\\
 \wtil{h}_{\ell}(x, q_{A, \frakp}  )&= \log_{A,\frakp, \ell}(x) \\
\wtil{h}_{\ell}(q_{A, \frakp} ,q_{A, \frakp'} )&=\begin{cases} \ell(q_{A, \frakp}) & \textrm{if } \frakp=\frakp'\\0 &\textrm{otherwise}\end{cases}
 \end{align*}
for all $ x$, $ y\in H^{1}_{f}(K, V_{p}A)$. 

\begin{lemm}\label{equivt}
 Suppose that $K/K_{0}$ is a finite Galois extension and that $A=E_{K}$ for an elliptic curve $E/K_{0}$. Then the pairing 
$$\wtil{h}\colon \wtil{H}^{1}_{f}(K, V_{p}A)\otimes \wtil{H}^{1}_{f}(K, V_{p}A) \to \Gamma_{K}\otimes L$$
is equivariant for the natural $\Gal(K/K_{0})$-action on all the terms.
\end{lemm}
\begin{proof} This is a special case of a result which holds in general for an induced representation $V={\rm Ind}_{K_{0}}^{K}V_{0}$, and which immediately follows, for example, from the cohomological construction of $\wtil{h}$ summarised in \cite[\S 0.16.0]{nek-selmer}.
\end{proof}

Finally, the third height pairing, the norm-adapted $h_{\ell}^{\rm norm}$ (following Schneider \cite{schneider}), is again defined on ${H}^{1}_{f}(K, V_{p}A)$; it is related to $h_{\ell}$ by 
\begin{align}
\label{h norm}
h^{\rm norm}_{\ell}(x,y)= h_{\ell}(x,y)- \sum_{\frakp \in S_{p}^{\rm exc}} {\log_{A,\frakp, \ell}(x) \log_{A,\frakp, \ell}(y) \over \ell(q_{A, \frakp})},
\end{align}
where each $\ell(q_{A, \frakp})\neq 0$ by
 \cite{st et}, and
for $\alpha_{1}, \alpha_{2},\alpha_{3}\in K_{\frakp}^{\times}$ with $\ell(\alpha_{3})\neq 0$, the ratio  $\ell(\alpha_{1}) \ell(\alpha_{2})/\ell(\alpha_{3})\in \Gamma\otimes \Q$  is defined as follows. 
  Let $\vpi\in K^{\times}_{\frakp}$ be an element of valuation $n> 0$ in the kernel of $\ell$; then  we may uniquely write $\alpha_{i}^{n}=\vpi^{r_{i}}u_{i}\beta_{i}$  for some $r_{i}\in\Z$, roots of unity $u_{i}\in \OO_{K,\frakp}^{\times}$,  and $\beta_{i}\in 1+\frakp \OO_{K, \frakp}^{\times}$. We can further write $\beta_{i}=\exp(b_{i})$ for $b_{i}\in \frakp\OO_{K, \frakp}$, and then define 
  $${\ell(\alpha_{1}) \ell(\alpha_{2})\over\ell(\alpha_{3})}:= {1\over np^{m}}\ell( \exp(p^{m}b_{1}b_{2}/b_{3}))\quad \text{in }\Gamma\otimes \Q$$
 for any sufficiently large $m\in\N$.

\begin{defi}\label{def-reg} Let $M$ be a finitely generated $\Z$-module, $L$  a  field of characteristic zero, $\Gamma $ a finite-dimensional $L$-vector space, 
 $h\colon M_{L}\otimes M_{L}\to \Gamma$ a symmetric $L$-bilinear form. The  \emph{regulator} of $h$ on $M$ is the discriminant 
$$R(M,h):=[M:\sum_{i=1}^{r}\Z x_{i}]^{-2}\cdot\det h(x_{i},x_{j})\quad \in {\rm Sym}^{r}\Gamma,$$
where $r:=\dim_{L} M_{L}$, the $x_{i}\in M$ are any $r$ elements forming a basis of $M_{L}$, and the determinant of an $r\times r$ matrix with entries in $\Gamma$ is computed via the usual alternating sum of products along generalised diagonals. 
\end{defi}

\subsection{Gross--Zagier and Waldspurger formulas} We give explicit versions of the Waldspurger, Gross--Zagier, and $p$-adic Gross--Zagier formulas.

\subsubsection{Shimura curves and Shimura sets} If $N^{-}$ is a squarefree integer, we denote by $B_{N^{-}}$ the quaternion algebra over $\Q$ of discriminant $N^{-}$; it is definite (resp. indefinite) if and only if $N^{-}$ is the product of an odd (resp. even) number of primes. If $R\subset B$ is an Eichler order, we denote by $X^{N^{-}}(R)$ the Shimura set (resp. Shimura curve) of level $R$. Explicitly, in the definite case $$X^{N^{-}}(R):=B^{\times}\bks\widehat{B}^{\times}/\widehat{R}^{\times}.$$
In the indefinite case, $X$ is a projective algebraic curve over $\Q$ such that 
$$X(\C)= B^{\times}\bks\mathfrak{H}^{\pm }\times \widehat{B}^{\times}/\widehat{R}^{\times} \cup \{{\rm cusps}\},$$
where $\mathfrak{H}^{\pm }=\C-\R$ and $ \{{\rm cusps}\}$ is a finite set, non-empty only if $N^{-}=1$. 

 If $N^{+}$ is an integer prime to $N^{-}$ and $R$ is an Eichler order of level $N^{+}$ which is either understood from context or whose specification is unimportant, we write  $X^{N^{-}}_{0}(N^{+})$ for $X^{N^{-}}(R)$.
 
\subsubsection{The setup} Let $E/\Q$ be an elliptic curve of conductor $N$ with ordinary reduction at the prime~$p$. Let $K$ be an imaginary quadratic field of discriminant $D$ prime to $N$, and  let $\eta$ be the associated quadratic Dirichlet character and $u:=|\OO_{K}^{\times}|/2$. We factor $N=N^{+}N^{-}$ where $N^{+}$ (respectively, $N^{-}$) is a product of primes which are split (respectively, inert) in $K$, and we assume that \emph{$N^{-}$ is squarefree}. We let 
$$\eps:=-\eta(N).$$

Let $B=B_{N^{-}}$,  fix an embedding $K\subset B$ and pick an Eichler order $R\subset B$ of level $N^{+}$ such that $R\cap K=\OO_{K}$. Let $X:= X^{N^{-}}(R)$.

Let $I_{E}$ be the ideal, in the spherical Hecke algebra ${\bf T}_{Np}$ for the level $Np$, generated by the  $T_{v}-a_{v}(E)$ for all $v\nmid Np$. (Note that, by the Jacquet--Langlands correspondence, the algebra ${\bf T}_{N}$  also acts on level-$N$ automorphic forms on any quaternion algebra.)

Finally, let $\phi$ be the normalised newform of level $\Gamma_{0}(N)$ associated with the isogeny class of  $E$, and let $(\phi, \phi)_{\Gamma_{0}(N)}$ be the Petersson norm with respect to the hyperbolic volume $dxdy$ on $\Gamma_{0}(N)\bks{\mathfrak{H}}$. 
We denote the ratio of the N\'eron period of $A=E_{K}$ with this Petersson norm by
$${c}_{\infty}({A}):={   \Omega_{A}\over  8\pi^{2}(\phi,\phi)_{\Gamma_{0}(N)}}
 \in \Q^{\times}.$$
 Note that this does not depend on the specific imaginary quadratic field $K$.

\begin{theo}[Waldspurger formula]\label{wald-f} Suppose that $\eps=+1$
(equivalently, that the squarefree integer $N^{-}$ is the  product of an odd number of primes). 
For each $x\in X=B^{\times}\bks\widehat{B}^{\times}/\widehat{R}^{\times}$ represented by $b_{x}\in \widehat{B}^{\times}$,  let $w_{x}:=|B^{\times}\bks  g_{x}\widehat{R}^{\times}g_{x}^{-1}/\{\pm1\}|$ and define a pairing on $\Z[X]$ by 
\begin{align}\label{pairing-B}
\langle f_{1}, f_{2}\rangle= \sum w_{x} f_{1}(x)f_{2}(x).
\end{align}
Let $\delta(f):=\langle f, f\rangle$ be the associated quadratic form.

Let $f\in \Z[X][I_{E}]$ be an integer-valued function on $X$ annihilated by the Hecke operators in  $I_{E}$,  and let 
$$p(f)=u^{-1}\cdot \sum_{t\in \Pic(\OO_{K})}f(t).$$
 Then  we have
$${L(A,1) \over |D_{K}|^{-1/2} \Omega_{A}}={{c}_{\infty}(A)^{-1}} \cdot { p(f)^{2}\over \delta(f)}.$$
\end{theo} 
\begin{proof} This is a special case of  \cite[Theorem 1.2]{cst}. The proof is based on the  original Waldspurger formula (\cite{wald} or \cite[(1.4.1)]{dd}), which is a more general if less explicit statement: namely $f$ is not necessarily a newform, and the right-hand side contains some extra local terms (toric integrals, denoted by $\beta_{v}$ in \cite{cst} and, for  comparisons, by $\alpha_{v}$ in \cite{yzz}, by $\mathcal{P}_{v}$ in \cite{CH}, by $Q_{v}$ in \cite{dd}). Then the deduction essentially  amounts to (i) a computation of the local integrals and (ii) a formula relating the Asai $L$-value to ${ 8\pi^{2}(\phi,\phi)_{\Gamma_{0}(N)}}={c}_{\infty}(A)^{-1}\Omega_{A}$.
\end{proof} 

\begin{theo}[Gross--Zagier formula]\label{gz}  Suppose that $\eps=-1$ (equivalently, that the squarefree integer $N^{-}$ is the  product of an even number of primes). Let $J$ be the Albanese variety of the Shimura curve $X$, and embed $X\into J$ by sending the Hodge class \cite{yzz} to $0$.
  Let $ f\colon J\to E$ be a nontrivial morphism, and let 
  $$\delta(f)=\deg(f):=f\circ f^{\vee}\in \End(E)=\Z.$$
Let $H$ be the Hilbert class field of $K$ and let  $P\in X(H)$ be a CM point for $K$ of conductor $1$. Consider the  Heegner point
$$P(f)=u^{-1}\cdot\sum_{t\in \Pic(\OO_{K})} f(P)^{\sigma_{t}}\in E(K)$$
where $t\mapsto \sigma_{t}$ is the reciproctiy map of class field theory. Then
$${L'(A,1) \over |D_{K}|^{-1/2} \Omega_{A}}={{c}_{\infty}(A)^{-1}} \cdot {h_{\rm NT}({P}(f), P(f)) \over \delta(f)}.$$
\end{theo} 
\begin{proof} This is a special case of \cite[Theorem 1.1]{cst}, which is deduced  from  the more general formula  of \cite[Theorem 1.2]{yzz} by the argument and calculations  recalled in the proof of Theorem \ref{wald-f}. 
\end{proof}

Before stating the $p$-adic analogue, we recall the existence of $p$-adic $L$-functions.

\begin{prop}\label{Lp for A} Let $E$ be an elliptic curve over $\Q$ of conductor $N$, let $K=\Q$ or an imaginary quadratic field of discriminant prime to $Np$,  and $\Gamma=\Gamma_{K}$. Let $A:=E_{K}$ and let $c$ be the generator of $\Gal(K/\Q)$. Then $(A, \Gamma) $ satisfies Hypothesis $(L_{p})$ from \S\ref{sec: the conj}, and the $p$-adic $L$-function $L_{p}(A)$ satisfies a functional equation relating $\chi$ to $\chi^{-c}$. If $K=\Q$ and $E$ has conductor $N$, the sign of the functional equation of $L_{p}(A)$ is 
\begin{align}\label{sgn}
\wtil{\eps}:=-\eta(N^{ (p)})\quad \in \{\pm 1\},
\end{align}
where $N^{(p)}$ is the prime-to-$p$ part of $N$ and $\eta$ is the quadratic Dirichlet character associated with $K$.
\end{prop}
\begin{proof} 
This is a theorem of Amice--V\'elu and Vishik  (see \cite[Chapter I]{mtt}) for $E$, and of Perrin-Riou \cite{PRL} for $E_{K}$ (for $p$ odd,  in general see \cite[Theorem A]{dd}; the proof of the functional equation in \cite{PRL} applies to the case $p=2$ given the existence of $L_{p}(A)$).
\end{proof}

\begin{theo}[$p$-adic Gross--Zagier formula]\label{pgz}  
Under the assumptions  of Theorem \ref{gz}, suppose moreover that $p$ splits in $K$,
and that $E$ has    ordinary  (good or  multiplicative) reduction at $p$.
Then
$${{\rm d}^{+}L_{p}(A,\one)}=\prod_{\frakp\vert p}e_{\frakp}(\one)\cdot{{c}_{\infty}(A)^{-1} } \cdot {h^{+}({P}(f), P(f)) \over \delta(f)}$$
in $\Gamma^{+}\otimes L$.
\end{theo} 
Note that when $E$ has split multiplicative reduction, the identity is the trivial $0=0$.
\begin{proof} This is a special case of the analogous result to \cite[Theorem 1.5]{cst}, which can be obtained by applying  word for word the  arguments of \emph{op. cit.} to \cite[Theorem B]{dd} instead of \cite[Theorem 1.2]{yzz}. When $E$ has good reduction and all $v\vert N$ split in $K$, this formula was proved by Perrin-Riou \cite{PR}.
\end{proof}

\subsubsection{The ``Heegner index''} Let $E$ be an elliptic  curve of conductor $N$.   Let $N^{-}$ be a squarefree divisor of $N$ and let  $f\in \Z[ X^{N^{-}}_{0}(N^{+})][I_{E}]$ (if $B_{N^{-}}$ is definite) or $f\colon J= {\rm Alb}\,(X^{N^{-}}_{0}(N^{+}))\to E$ (if $B_{N^{-}}$ is indefinite).
 Let $K$ be an imaginary quadratic field, and 
define
\begin{align}\label{index}
I(E_{K}, f):= |\Sha(E_{K})|_{\rm an}\cdot c_{\infty}(E_{K})  {\prod_{w\vert N} c_{w}(E_{K})}\cdot \delta(f),
\end{align}
where $w$ runs over the primes of $K$.  Under Hypothesis ($\textup{BSD}_{\infty}$)--2 for $E_{K}$, we have $I(E_{K}, f)\in\Q^{\times}$.

\begin{coro}\label{index-heeg} Let $E/\Q$ be an elliptic curve of conductor $N$, let $K$ be an imaginary quadratic field, $A=E_{K}$, and   assume that $r_{\rm an}:={\rm ord}_{s=1}L(A, s)\leq 1  $.  Let $N^{-}$ and $f$ be as in either Theorem \ref{wald-f} (case $r_{\rm an}=0$)  or Theorem  \ref{gz} (case $r_{\rm an}=1$). Then there is a positive integer $i(A, f)$ such that $$i(A, f)^{2}=I(A, f),$$ 
and explicitly
$$i(A, f)=\begin{cases}  |A(K)|\cdot |p(f)| &\text{if } r_{\rm an}=0\\
[A(K):\Z P(f)]  &\text{if } r_{\rm an}=1,
\end{cases}$$
where in both cases  the right-hand  is finite by the work of Kolyvagin \cite{koly}.
\end{coro}
\begin{proof}  The result  follows form the definitions and  the Waldspurger and Gross--Zagier formulas   (Theorems \ref{wald-f} and \ref{gz}). 
\end{proof}

\section{Evidence over $\Q$}\label{evidenceQ}
In this section we prove Theorem \ref{mainQ}.

\subsection{Preliminaries} 
We first study some invariance properties of Conjecture {$({\textup{BSD}}_{p})$}.

We start by  recording an alternative statement  of our conjecture, closer to  the original conjecture of Mazut--Tate--Teitelabum \cite{mtt} in its usual formulation.
Recall the $\mathcal{L}$-invariant defined in \eqref{Linv}, and let 
$$L^{*}_{\rm alg}(A, 1):= {L^{(r)}(A, 1)\over r! |D_{K}|^{-1/2}  \Omega_{A}R_{\rm NT}(A)},$$
where  $r={\rm ord}_{s=1}L(A,s)$.

\begin{prop}\label{compatibility} Let $A/K$ be an elliptic curve with ordinary reduction at all the primes above $p$. Suppose $(A, \Gamma)$ satisfies Hypotheses {$(L_{p})$} and {$({\textup{BSD}}_{\infty})$}--\ref{BSD1}-\ref{BSD2} from \S\ref{sec: the conj}.
  Let $ r={\rm ord}_{s=1}L(A,s)$, $\wtil{r}:= r+|S_{p}^{\rm exc}|$. 
Then Conjecture {$({\textup{BSD}}_{p})$} holds for  $(A, \Gamma)$ if and only if 
$${\rm d}^{\wtil{r}} L_{p}(A^{(\Gamma_{\Q})}, \one)= \prod_{\frakp\in  S_{p}-S_{p}^{\rm exc}} e_{\frakp}
(\one)\prod_{\frakp\in S_{p}^{\rm exc}}  \mathcal{L}_{\frakp, \ell_{\Gamma}}(A)\cdot R^{\rm norm}_{\ell_{\Gamma}}(A)\cdot L_{\rm alg}^{*}(A,1)$$
in ${\rm Sym}^{\wtil{r}} \Gamma\otimes L$, 
with $\ell_{\Gamma}\colon K^{\times}\bks K_{\A^{\infty}}^{\times}\to\Gamma_{K}\to \Gamma$.
\end{prop}
The proof, by elementary linear algebra, is already given in \cite[p. 35]{mtt} in slightly different language. 
\begin{lemm}\label{splitinert} Suppose that $E/\Q$ has  multiplicative reduction at $p$ and let $K$ be an imaginary quadratic field in which $p$ is unramfied. The twist $E^{(K)}$  has multiplicative reduction at $p$ too, and the following are equivalent:
\begin{enumerate}
\item $p$ splits in $K$ (respectively, $p$ is inert in $K$);
\item $E$ and $E^{(K)}$ have  the same reduction type at $p$, either both split or both non-split (respectively, $E$ and $E^{(K)}$ have different reduction type at $p$);
\item the base-change $E_{K}$ has zero or two (respectively, one) primes above $p$ of split multiplicative reduction.
\end{enumerate}
If $p$ splits in $K$, then more precisely $r^{\rm exc}(E_{K})=2$  (respectively, $r^{\rm exc}(E_{K})=0$) if $E$ has split (respectively non-split) multiplicative reduction.
\end{lemm}
The proof is easy and left to the reader.

\begin{prop}\label{invar} Let $E$  be an elliptic curve over a number field $F$ and let $\Gamma$ be a $\Z_{p}$-free quotient of $\Gamma_{F}$. Let ``Conjecture X'' be either of Hypotheses $(L_{p})$, {$({\textup{BSD}}_{\infty})$}, or Conjecture {$({\textup{BSD}}_{p})$}.
\begin{enumerate}
\item If $E'$ is isogenous to $E$, then  Conjecture X 
holds for $(E, \Gamma)$ if and only if it holds for $(E', \Gamma)$.

\item Suppose that $F=\Q$ and let $K$ be an imaginary quadratic field in which $p$ is unramified, and $\Gamma=\Gamma_{\Q}$ (which we  also view as a quotient of $\Gamma_{K}$); then  Hypothesis $(L_{p})$ holds for $E$, the twist $E^{(K)}$, and the base-change $E_{K}$.  

 \noindent Hyptothesis  {$({\textup{BSD}}_{\infty})$}--\ref{BSD1}-\ref{BSD2}  holds for $(E, \Gamma)$ and $(E^{(K)}, \Gamma)$ if and only if it holds for $(E_{K}, \Gamma)$. If this is the case, then:
 \begin{enumerate}
\item if Conjecture {$({\textup{BSD}}_{p})$}  holds for $(E, \Gamma)$ and $(E^{(K)}, \Gamma)$, then it holds for $(E_{K}, \Gamma)$. 
\item if Conjecture {$({\textup{BSD}}_{p})$}  holds for $(E_{K}, \Gamma)$ and  $(E^{(K)}, \Gamma)$ and the  two sides of \eqref{mainid} for $(E^{(K)}, \Gamma)$ are nonzero, then Conjecture {$({\textup{BSD}}_{p})$} holds for $(E, \Gamma)$.
\end{enumerate}
\end{enumerate}
\end{prop}
\begin{proof} Part 1 is obvious for $(L_{p})$, and classical for  {$({\textup{BSD}}_{\infty})$}, see \cite{tate-bsd}. It is already observed in \cite{mtt} that the proof can be adapted to the case of  {$({\textup{BSD}}_{p})$} by noting that, for any $\frakp\in S_{p}^{\rm exc}(E)=S_{p}^{\rm exc}(E')$, the $\Q$-lines generated by $q_{E,\frakp}$ and $q_{E', \frakp}$ in $F_{\frakp}^{\times}$ are the same, and the right-hand side of \eqref{mainid} for $E$ only depends on $\Q q_{E, \frakp}$ and not on $q_{E,\frakp}$. (In the formulation of Proposition \ref{compatibility}, this is simply  the observation that $\mathcal{L}_{\frakp}(E)=\mathcal{L}_{\frakp}(E')$.)

The assertions on $(L_{p})$ in part 2 are contained in Proposition \ref{Lp for A}. 
The equivalence statement
for   {$({\textup{BSD}}_{\infty})$} is also classical. To adapt its proof to the case of   {$({\textup{BSD}}_{p})$}, note that by Lemma \ref{splitinert}, the sets
$S_{p}^{\rm exc}(E_{K})$ and $S_{p}^{\rm exc}(E)\coprod S_{p}^{\rm exc}(E^{(K)})$ are (non-canonically) in bijection, and if $\frakp\in S_{p}^{\rm exc}(E_{K})$ then $q_{E_{K}, \frakp}=q_{E_{0}, p}$ for some $E_{0}\in \{E, E^{(K)}\}$.
\end{proof}

\subsection{Proof of Theorem \ref{mainQ}}
We prove Theorem \ref{mainQ}, whose statement we recall in the following slightly more general form.
\begin{theo}\label{Q} Let $E/\Q$ be an elliptic curve with ordinary reduction at the prime  $p$. Suppose that $r:={\rm ord}_{s=1}L(E, s)\leq 1$.
Then:
\begin{enumerate}
\item\label{nspl}  If the reduction of $E$ at $p$ is \emph{not} split multiplicative or $r=0$, then  Conjecture {$({\textup{BSD}}_{p})$} holds.
\item If $r=1$ and the reduction of $E$ at $p$ is split multiplicative, 
suppose that $p\geq 5$. Then Conjecture {$({\textup{BSD}}_{p})$} holds up to a nonzero rational number; if moreover there is at least another  prime $m\neq p$ of multiplicative reduction for $E$, then Conjecture {$({\textup{BSD}}_{p})$} holds exactly, that is
$${\rm d}^{2} L_{p}(E, \one)=    \mathcal{L}_{p}(E)\cdot R^{\rm norm}(E) \cdot L_{\rm alg}'(E,1)$$
in $\Gamma_{\Q}^{\otimes 2}\otimes\Q$.
\end{enumerate}
\end{theo}

\begin{rema} If $E/\Q$ has good supersingular reduction at $p$ and $r=1$, then the analogue of Conjecture {$({\textup{BSD}}_{p})$} is proved by Kobayashi \cite{kobayashi}.
\end{rema}

\subsubsection{Proof in case \ref{nspl}}
If $r=0$, then the result is trivial unless the reduction is split multiplicative; in that case, it is due to Greenberg--Stevens \cite{GS} if $p\geq 5$. A different proof, based on ideas of Kato--Kurihara--Tsuji (unpublished),   is given by Kobayashi in \cite{kob-mtt} (see also \cite{colmez-bsdp}); as written there it applies directly to any $p\geq 3$, and it extends to cover the case of $p=2$ after inserting the appropriate  modifications of the theory of the Coleman map described in \cite[\S3]{kobayashi} and \cite{ota}.

Suppose that $r=1$ and the reduction is not split multiplicative. By Lemma \ref{invar} we may in fact prove the formula for $E_{K}$, where we choose $K$ to be  an imaginary quadratic field in which all primes dividing $N$ split, and such that $L(E^{(K)},1)\neq 0$. (The existence of such $K$ is guaranteed by \cite{murty2}.)
Let $A=E_{K}$, and use the `$\pm$' notation introduced before Theorem \ref{exc}. Then by Corollary \ref{index-heeg} we may rewrite the $p$-adic Gross--Zagier formula of Theorem \ref{pgz} as 
$${{\rm d}^{+}L_{p}(A,\one)}=\prod_{\frakp\vert p}e_{\frakp}(\one) \cdot R^{+}(A)\cdot L_{\rm alg}^{*}(A,1),$$
as desired.\footnote{This argument was of course  already made by Perrin-Riou \cite{PR} when $E$ has good reduction.}

\begin{rema}\label{rem dd} If $E/F$ is a modular elliptic curve over a totally real field, whose reduction at every $\frakp \vert p$ is ordinary but not split multiplicative,  and if we have ${\rm ord}_{s=1}L(E,s)\leq 1$, then  by the same argument (comparing the main results of \cite{dd} and \cite{yzz}),  Conjecture {$({\textup{BSD}}_{p})$} holds for $E$ up to  a conjecture on the commensurability of  $\Omega_{E}$ with a certain  automorphic period (this last point is discussed in \cite[\S9]{dd1}). 
\end{rema}
\begin{rema}\label{rem spiess} The result of Greenberg--Stevens is generalised to modular elliptic curves over totally real fields by Mok \cite{mok} and Spiess \cite{spiess}. In particular this proves Conjecture {$({\textup{BSD}}_{p})$} over totally real fields in the case of analytic rank~$0$, up to the conjecture on periods mentioned in the previous remark. 
\end{rema}

\subsubsection{Proof in case 2} The rest of this section is dedicated to the proof of Theorem \ref{Q} in the case in which $E$ has split multiplicative reduction at $p\geq 5$ and $r:={\rm ord}_{s=1} L(E,s)=1$. 
The result is  proven by Venerucci \cite{venerucci} up to a rational constant. We remove the ambiguity assuming   the condition that $E$ has 
 a  prime $m\neq p$ of multiplicative reduction; we fix such an $m$.
 We will in fact prove Conjecture {$({\textup{BSD}}_{p})$}   for the base-change of  $E$ to a suitable auxiliary imaginary quadratic field $K$. More precisely, denoting by $N$ the conductor of $E$, let $K$ be an imaginary quadratic field of discriminant $D$  such that (a) $p$ and $m$ are inert in $K$, (b) every prime divisor of $N/pm$ splits in $K$, (c) $L(E^{(K)}, 1)\neq 0$, (d) $\OO_{K}^{\times}=\{\pm 1\}$. (The existence of such  $K$ is again guaranteed by   \cite{murty2}.) By Proposition \ref{invar}
and the (trivial) validity of Conjecture {$({\textup{BSD}}_{p})$} for $E^{(K)}$, the statements of Conjecture {$({\textup{BSD}}_{p})$} for $E$ and for $E_{K}$ are equivalent.

We need some more notation. 
Let $J$ be the Jacobian of the Shimura curve $X=X^{mp}_{0}(N/mp)$.
 Let $I_{E}$ be the ideal in the spherical Hecke algebra generated by the operators $T_{v}-a_{v}(E)$ for $v\nmid N$, and let $f'\colon X'=X_{0}^{m}(N/m)\to \Z$ 
 be a generator of the space of integer-valued functions annihilated by $I_{E}$.
\begin{prop}[Venerucci]\label{venD} 
With notation as above, suppose that $E$ is an optimal quotient $f\colon J\to E$ of $J$, that is, that  $f^{\vee}$ is injective.\footnote{Note that this can always be achieved up to replacing $E$ by an isogenous curve.}
Let  $P(f)\in E(K)$ be the corresponding Heegner point.

Then
$$({\rm d}^{+})^{2} L_{p}(E_{K}, \one) =  {c}_{\infty}(E_{K})^{-1}\, \mathcal{L}_{p\OO_{K}}^{+}(E_{K})
{ h^{\rm norm,+}(P(f),P(f))\over \delta(f')}
\cdot
c_{p}(E) ,$$
where $h^{\rm norm, +}=h_{\ell^{+}}^{\rm norm}$ (with the notation of Theorem \ref{exc}), and $c_{p}$ is the Tamagawa number of $E$ at $p$.
\end{prop}
\begin{proof} 
This is \cite[Theorem D]{venerucci}, whose ${\bf P}$ is our $P(f)$, taking into account  the value of the constant $\ell_{3}$ of \emph{loc. cit.}, which is explicitly given in terms of previously defined constants $\ell_{2}$ and $\ell$ in \cite[\S6.5, \S6.4, Remark 2.2]{venerucci}. Note that in fact $\ell_{3}=2\ell_{2}\cdot {\rm ord}_{p}(q_{E,p})$ (there is a typo in \cite{venerucci}), and that by Tate's $p$-adic uniformisation, we have $c_{p}
(E)={\rm ord}_{p}(q_{E,p})$.

We explain the different normalisations accounting for apparent discrepancies in the powers of $2$ between \cite{venerucci} and our statement. As the composition $\Gamma_{\Q}\into \Gamma_{K}\stackrel{\ell^{+}}{\to} \Gamma^{+}\cong \Gamma_{\Q}$ is multiplication by $2$, we have  $h^{\rm norm, +}(P(f),P(f))= 2h^{\rm norm}_{\ell_{\Q}}(P(f),P(f))$, where in the second expression we view $P(f)\in E(\Q)_{\Q}$; similarly, $\mathcal{L}_{p\OO_{K}}^{+}(E_{K})=2\mathcal{L}_{p}(E)$.\footnote{This  factor of $2$ can also be viewed as the interpolation factor for $L_{p}(E^{(K)})$.} Secondly, our $({\rm d}^{+})^{2}$ equals  $1/2$ times the operator ${d^{2} \over ds^{2}}$ of \cite[Theorem D]{venerucci} (after fixing $\Gamma^{+}\cong \Gamma_{\Q}\cong \Z_{p}$).
Finally, the Petersson product on $\GL_{2}$ used in \cite[Remark 2.2]{venerucci}, is $(1/2)\cdot 8\pi^{2}$ times our Petersson product; and the quaternionic inner product of \emph{loc. cit.} is $1/2$ times our inner product. Indeed both of the normalisations in \emph{loc. cit.} are inherited from   \cite{bd-hida}: see \cite{bd-hida} for the first assertion, whereas  it is
easiest
 to see the  second assertion by comparing the Waldspurger formulas of \cite[Proposition 3.4]{bd-hida} and Theorem  \ref{wald-f}.
\end{proof}

Comparing Proposition \ref{venD} with Corollary \ref{index-heeg},
  Conjecture {$({\textup{BSD}}_{p})$}  for $(E_{K}, \Gamma_{\Q})$  is reduced  to the following statement.
\begin{prop}\label{tak}
With the notation of Proposition \ref{venD}, we have
$${\delta(f)} ={ \delta( f') \over c_{p}(E)}.$$
\end{prop}
For future reference, we remark that the result and its proof remain valid  if ${m}$ is replaced by any squarefree product $N^{-}$ of an odd number of  primes different from $p$.
\begin{proof} Let $\X_{p}(J)$ be the character group of the toric part of the reduction of $J$.
The  asserted result  is proved by Takahashi \cite[Corollary 2.6]{takahashi} (see also \cite[\S2]{rib-tak}) with  $\delta(f')=\langle f', f'\rangle$ replaced by $u_{J}(g,g)$, where $g$ is a generator of the free rank $1$ saturated $\Z$-submodule $\X_{p}(J)[I_{E}]\subset \X_{p}(J)$ annihilated by $I_{E}$, and $u_{J}\colon \X_{p}(J)\times \X_{p}(J)\to \Z$ is Grothendieck's monodromy pairing.

To complete the proof, we then need to compare $\langle f', f'\rangle$ and $u_{J}(g,g)$. In fact  by the Cerednik--Drinfeld 
uniformisation, $\mathscr{X}_{p}(J)$ is a saturated submodule of $\Z[\mathscr{E}]$ where ${\mathscr{E}}=B_{m\cdot\infty}^{\times}\bks \widehat{B}_{m\cdot\infty}^{\times}/\widehat{R}^{\times}$ is the set of edges in the dual graph of the reduction of the Shimura curve $X$ at $p$, so that $g$ may be identified with $f'$ up to a sign. Moreover  by the Picard--Lefschetz formula, the pairing $u_{J}$ is the restriction of the pairing $\langle\, , \, \rangle$ on $\calE$ of \eqref{pairing-B}.  We conclude that $u_{J}(g,g)= \langle f', f'\rangle=\delta(f)$ as desired.
\end{proof}

This completes the proof of    Conjecture {$({\textup{BSD}}_{p})$}  for $(E_{K}, \Gamma_{\Q})$. By Proposition \ref{invar} and the trivial validity of  {$({\textup{BSD}}_{p})$} for $E^{(K)}$, we deduce {$({\textup{BSD}}_{p})$}  for $E$, completing the proof of Theorem \ref{Q}.

\section{Anticyclotomic theory}\label{sec: anti}

The following framework and notation will be adopted for the rest of the paper (unless otherwise noted). We let $E/\Q$ be an elliptic  curve of conductor $N$ with ordinary reduction at $p$, and $K$ an imaginary quadratic field of discriminant $D$ prime to $Np$. We factor 
 $N=N^{+}N^{-}$ where $N^{+}$ (respectively $N^{-})$ is a product of primes which are split (respectively inert) in $K$, and we  assume that $N^{-}$ is squarefree.
  For $N^{?}=N, N^{+}, N^{-}$ we denote by $N^{?,(p)}$ the prime-to-$p$ part of $N^{?}$.  Recall also the sign   $\wtil{\eps}:=-\eta(N^{ (p)})$ from \eqref{sgn}.
  
For primes $v\nmid N$, we let $a_{v}:=v+1-|E({\bf F}_{v})|$, and we let $\alpha\in \Z_{p}$ be the unit root of $X^{2}-a_{p}X+p $ if $p\nmid N$, and $\alpha=+1 $ (resp. $\alpha=-1$) if $E$ has split (resp. nonsplit) multiplicative reduction at $v$. We denote by $I_{E}\subset {\bf T}^{Np}$ the ideal generated by $T_{v}-a_{v}(E)$ for $v\nmid Np$.

   We will  use the notation $\Gamma^{\pm}$, $\ell^{\pm}$, $\mathcal{L}_{\frakp}^{\pm}(A)$, $R^{\pm}(A)$ introduced before the statement of Theorem \ref{exc}, as well as $\wtil{R}^{\pm}(A):=\wtil{R}_{\ell^{\pm}}(A)$.  We also write $\Gamma=\Gamma_{K}$ and $\Y=\Y_{\Gamma}=\Y^{+}\times \Y^{-}$ where the ring of bounded functions on $\Y^{\pm}$ is $\Z_{p}\llb \Gamma^{\pm}\rrb_{\Q_{p}}$.

\subsection{Theta elements}\label{sec: theta}

We recall the construction of \emph{theta elements}, certain functions on $\Y^{-}$ interpolating toric periods (case $\wtil{\eps}=+1$)  and Heegner points (case $\wtil{\eps}=-1$); the reader is referred to the original works of Bertolini--Darmon \cite[\S2]{bdMT} or \cite[\S4]{bdsurvey}  for more details.

 Let $B=B_{N^{-,(p)}}$ be the quaternion algebra of discriminant $N^{-,(p)}$.   Let $R\subset R_{0}\subset B$ be Eichler  orders of respective conductors  $pN^{+, (p)}$ and $N^{+, (p)}$. We may choose an isomorphism $B_{p}\cong M_{2}(\Q_{p})$ identifying $R^{\times}$ with  the usual congruence subgroup $\Gamma_{0}(p)$, and identifying $R_{0}^{\times}$ with $\Gamma_{0}(1)$ (resp. $\Gamma_{0}(p)$) if $p\nmid N$ (resp. $p\vert N$).  Let $X:= X^{N^{-,(p)}}(R)$,  $X_{0}:=X^{N^{-,(p)}}(R_{0})$ be the associated Shimura sets (case $\wtil{e}=+1$) or Shimura curves  (case $\wtil{e}=+-1$). Fix an embedding $K\subset  B$ such that $R\cap  K=\OO_{K}$.

For $n\geq 0$, let $\OO_{K,n}:= \Z_{p}+p^{n}\OO_{K}$ and let $G_{n}:={\rm Pic}(\OO_{K,n})$; let $G_{\infty}:=\varprojlim_{n}  G_{n}$  (then   $\Gamma^{-}$ is the $\Z_{p}$-free quotient of $G_{\infty}$).
 Let  $H_{n}\subset K^{\rm ab}$ be the ring class field of $K$ of conductor $n$, $H_{\infty}:=\bigcup_{n}H_{n}$. By class field theory, we have a canonical identification $G_{n}\cong \Gal(H_{n}/K)$ for all $n=0,1, \ldots, \infty$.

For $n\geq 1$, let
\begin{align}\label{gn}
g_{n}:=  \smalltwomat {p^{n}} 1{}1\in \GL_{2}(\Q_{p})=B_{p}^{\times}.
\end{align}

\subsubsection{Case  $\wtil{\eps}=+1$} 
  Let $z_{0}\in X$ be the image of $1\in K^{\times}\subset B^{\times}\subset \widehat{B}^{\times} $ and,  recalling that $ B_{p}^{\times}\subset \widehat{B}^{\times}$ acts on $X$,  let $z_{n}:=z_{0}\cdot g_{n}\in X$.  Let $f\colon X_{0}\to \Z$ be
   a nontrivial function annihilated by the Hecke operatros in $I_{E}$. Denote still by $f\colon X\to X_{0}\stackrel{f}{\to} \Z$ the composition, and let $f^{\dagger}\in \Z_{p}[X]$ be $f^{\dagger}:=f  $ if
$p\vert N$, and $f^{\dagger}:=f-\alpha^{-1}\cdot \smalltwomat 1{}{}p f$ if $p\nmid N$, where $\smalltwomat 1{}{}p\in B_{p}^{\times}$ acts on $\Z[X]$ via its right action on $X$. 

For $n\geq 0$, let
 $${{\Theta}}_{n}={\Theta}_{n}(f):= u_{n}^{-1}\cdot\sum_{t\in G_{n}}\alpha^{-n}f^{\dagger}(tz_{n})\otimes [t] \quad\in 
\Z_{p}[G_{n}],$$
where $u_{0}=u=|\OO_{K}^{\times}|/2$ and $u_{n}=1$ if $n\geq 1$. 
The elements $\Theta_{n}$   satisfy the compatibility relation $\Theta_{n}\mapsto \Theta_{n-1}$ for all $n\geq 1$ under the natural  quotient maps. Hence we may form their limit in $\Z_{p}\llb G_{\infty}\rrb$ and consider its projection to $\Z_{p}\llb \Gamma^{-}\rrb$:
$$\Theta={\Theta}(f):=\lim_{n} {\Theta}_{n}\in\Z_{p}\llb \Gamma^{-}\rrb\subset \OO(\Y^{-})^{\rm b}.$$

(Note that taking the image of $z_{n}$ under $\alpha^{-n}f^{\dagger}$ is the same as taking the image under $f$ of the `regularised' versions of $z_{n}$ defined in \cite[\S 2.5]{bdMT}.)

\begin{rema}\label{comparison}
 The construction of $\Theta$ presented  in \cite[\S4]{bdsurvey} depends on the choice of an `admissible' (one or two choices have to be excluded) half-line $\{v_{0}g_{n}\}_{n}$ in the Bruhat--Tits tree ${\mathscr{T}} $ of $\GL_{2}(\Q_{p})=B_{p}^{\times}$, originating from a vertex $v_{0}$ corresponding to $z_{0}$. Here we have made the specific choice $g_{n}= \smalltwomat {p^{n}}{1}{}1$. It is explained in \emph{loc. cit.} that the group $K_{p}^{\times}/\Q_{p}^{\times}$ acts transitively on the set of admissible half-lines, so that another choice -- such  as the one made in \cite{CH} -- would yield the element $\Theta'=[t_{p}] \Theta$ for some $t_{p}\in K_{p}^{\times}/\Q_{p}^{\times}$ with image $[t_{p}]\in \Gamma^{-}$. 
It follows that the following are independent of the choice of half-line in ${\mathscr T}$:  (i) the `leading term'  of $\Theta \in\OO(\Y^{-})^{\rm b}$ at $\chi^{-}=1$ (and more generally  its image in $\calI_{\Y^{-}}^{r}/\calI_{\Y^{-}}^{r+1}$, for any $r$ such that $\Theta \in \calI_{\Y^{-}}^{r}$); (ii) letting $*\colon \Z_{p}
\llb \Gamma^{-}\rrb\to \Z_{p}\llb \Gamma^{-}\rrb$ be the involution induced by inversion on $\Gamma^{-}$, the element $\Theta\cdot \Theta^{*}$.
\end{rema}

\subsubsection{Case  $\wtil{\eps}=-1$}
  Let $J$ (resp. $J_{0}$) denote the  Albanese variety of $X$ (resp. $X_{0}$). Let $f\colon J_{0}\to E$ be a nontrivial morphism  and denote still by $f$ the composition $J\to J_{0}\stackrel{f}{\to} E$.  
 Let $f^{\dagger }\in \Hom(J, E)\otimes \Z_{p}$ be $f^{\dagger}:=f$ if  $p\vert N$ and $f^{\dagger}:=f-\alpha^{-1}\cdot \smalltwomat 1{}{}p f$, where $\smalltwomat 1{}{}p\in B_{p}^{\times}$ acts on $\Hom (J, E)$ via its action on $J$.

Let $z_{0}\in X(H_{0})$ be a CM point of conductor $1$ and, recalling that $\GL_{2}(\Q_{p})\cong B_{p}^{\times}\subset \widehat{B}^{\times}$ acts on $X$, let
 $$z_{n}:=z_{0}\cdot g_{n}\in X(H_{n}),$$
 a CM point of conductor $p^{n}$ in $X$.

For $\chi\colon G_{\infty} \to \Lambda^{\times}$ a character valued in the units of some ring $\Lambda$, let $\Lambda(\chi)$ denote the Galois-module $\Lambda$ with action by $\chi$, and let $$A(\chi):= (A(H_{\infty})\otimes \Lambda(\chi))^{ G_{\infty}}.$$
If ${G}$ is any quotient of $ G_{\infty}$ and $\Lambda=\Lambda_{{G}}=\Z_{p}\llb{G}\rrb$, denote by $\chi_{{\rm univ}, {G}}\colon \Gal(H_{\infty}/K)\to \Lambda_{{G}}^{\times}$ the universal character; when ${G}=G_{n}$ we set $\chi_{{\rm univ},n}= \chi_{{\rm univ}, {G}} $; when ${G}=\Gamma^{-}$ we  set $\chi_{{\rm univ}}^{-}:=\chi_{{\rm univ}, \Gamma^{-}}$

For $n\geq 0$, let
 $${\mathscr{P}}_{n}={\mathscr{P}}_{n}(f):= u_{n}^{-1}\cdot\sum_{\sigma\in G_{n}}\alpha^{-n}f^{\dagger}(z_{n}^{\sigma})\otimes [\sigma]\quad  \in A(\chi_{{\rm univ}, n}),$$
where again $u_{0}=u=|\OO_{K}^{\times}|/2$ and $u_{n}=1$ if $n\geq 1$. 
The elements ${\mathscr P}_{n}$ satisfy the compatibility relation ${\mathscr P}_{n}\mapsto {\mathscr P}_{n-1}$ for all $n\geq 1$ under the natural quotient maps $A(\chi_{{\rm univ}, n})\to A(\chi_{{\rm univ}, n-1})$. Hence we may form their limit in $A(\chi_{{\rm univ}, G_{\infty}})$ and consider its projection to $A(\chi^{-}_{{\rm univ}})$:
$$\mathscr{P}=\mathscr{P}(f):=\lim_{n} \mathscr{P}_{n}\in A(\chi_{\rm univ}^{-}).$$

We still denote by $\mathscr{P}$ the image of $\mathscr{P}$ in the Selmer group $H^{1}_{f}(K, V_{p}A\otimes \Z_{p}\llb \Gamma^{-}\rrb)$.  By the work of \nek \ \cite{nek-selmer} there is a natural extension
 \begin{align}\label{bigtil}
\wtil{H}^{1}_{f}(K, V_{p}A\otimes \Z_{p}\llb \Gamma^{-}\rrb)\to H^{1}_{f}(K, V_{p}A\otimes \Z_{p}\llb \Gamma^{-}\rrb);
\end{align} 
the  specialisation of $\wtil{H}^{1}_{f}(K, V_{p}A\otimes \Z_{p}\llb \Gamma^{-}\rrb)$ under the augmentation map $\Z_{p}\llb \Gamma^{-}\rrb\to \Z_{p}$ is  $\wtil{H}^{1}_{f}(K, V_{p}A)$
(and in fact a similarly good behaviour holds under all specialisations). By
(0.8.0.1) \emph{ibid.} and the fact that $\chi_{\rm univ}^{-}$ is infinitely ramified,\footnote{See \cite[\S1.4]{dd} for some more details.} the map \eqref{bigtil}
is an isomorphism, and we still denote by 
$$\mathscr{P}\in \wtil{H}^{1}_{f}(K, V_{p}A\otimes \Z_{p}\llb \Gamma^{-}\rrb)$$
the resulting element.

\begin{rema} A more explicit  but equivalent construction of the lifting of $\mathscr{P}$ from $A(\chi^{-}_{\rm univ})$ to $\wtil{H}^{1}_{f}(K, V_{p}A\otimes \Z_{p}\llb \Gamma^{-}\rrb)$ (or rather, in the framework adopted there, to a certain  extended big Mordell--Weil group) is given in \cite[\S2.6]{bdMT}. 

The observations of Remark \ref{comparison} also apply to the present case.
\end{rema}

\subsection{Bertolini--Darmon conjectures}  We  review (and  slightly reformulate) conjectures of Bertolini and Darmon \cite{bdMT, bdsurvey} on the `leading terms' of $\Theta$ and ${\mathscr P}$. 

\paragraph{Pfaffian regulators} Let $M$ be a finitely generated $\Z$ module equipped with an involution $c$. Let $L$ be a field of characteristic zero,  $\Gamma$ a finite-dimensional vector space over $L$, and   $h\colon M_{L}\otimes M_{L}\to \Gamma$ a symmetric bilinear form satisfying 
\begin{equation}\label{h-feq}
h(c x, c y)=- h(x,y)
\end{equation}
for all $x,y\in M$. Denote by $M_{L}^{\pm}$ the $\pm$-eigenspaces for the action of $c$ on $M_{L}$ and, for $?=+, -, \emptyset$, let $r^{?}:=\dim M_{L}^{?}$. 

 Suppose first that $r$ is even.
   We define the  \emph{Pfaffian-regulator} of $h$ on $M$  to be 
$${\rm pf}\,(M,h):=[M:\sum \Z x_{i}]^{-1}\cdot {\rm pf } \,((h(x_{i}, cx_{j}))_{i,j=1}^{r}) \quad \in {\rm Sym}^{r/2}\Gamma$$
where $x_{1}, \ldots, x_{r}$ are elements of $M$ forming a basis of $M_{L}$, and in the right-hand side `${\rm pf}$' denotes the Pfaffian of an  antisymmetric matrix. Note that ${\rm pf}\,(M)$ is well-defined only up to a sign, depending on the orientation of the chosen basis. Comparing with Definition \ref{def-reg}, we find
${\rm pf}\,(M)^{2}={\rm sgn}(c)R(M),$
where $R(M,h)$ is the regulator of $h$ on $M$ and ${\rm sgn}(c)$ is the determinant of $c$ on $M_{L}$. In fact $h$ is degenerate unless $r^{+}=r^{-}$; therefore in all cases
\begin{align}\label{pf2}
{\rm pf}\,(M,h)^{2}=(-1)^{r/2}R(M,h).\end{align}

If $r=\dim M_{L}$ is odd, then $R(M,h)=0$. In this case we can define a `higher' Pfaffian regulator ${\rm Pf}\,(M,h)\in M\otimes {\rm Sym}^{(r-1)/2}\Gamma$.
Suppose first that there is a non-torsion element $x\in M$  in the radical of $h$ whose image in $M_{L}$ belongs to the larger of $M_{L}^{\pm}$.
Then 
$${\rm Pf}\,(M,h):= x\otimes {\rm pf}\,(M/\Z x) \quad \in M_{L}\otimes {\rm Sym}^{(r-1)/2}\Gamma.$$
This is well-defined up to sign: it is apparent that the definition only depends on the line generated by $x$ in $M_{L}$; if this is not unique, then $|r^{+}-r^{-}|\geq 2$, which implies $R(M)=0$, and ${\rm Pf}\, (M)=0$ regardless of the choice of $x$ in the definition. 

In general, denote by $\baar{M}$ the image of $M$ in $M_{L}$, and  let $\gamma\in \GL(M_{L})$ be an element such that $\gamma \baar{M}\subset M_{L}$ contains an element $x'$ in the radical of $h$;  let $M':=\gamma \baar{M}\oplus M_{\rm tors}$. Then we define
\begin{align}\label{pfff}
{\rm Pf}\,(M,h):=\det(\gamma)^{-1}{\rm Pf}\, (M', h).
\end{align}

\medskip 

Going back to our situation, consider   the $\Z$-module $A^{\dagger}(K)$
 equipped with the pairing $\wtil{h}^{-}$. If $\dim A^{\dagger}(K)_{\Q}$ is even, we let $${\rm pf}^{-}(A):= {\rm pf}\, (A^{\dagger }(K),\wtil{h}^{-}).$$ 
 If $\dim A^{\dagger}(K)_{\Q}$ is odd, we let $${\rm Pf}^{-}(A):= {\rm Pf}\, (A^{\dagger }(K), \wtil{h}^{-}).$$

Recall the number $I(A,f)$ from  \eqref{index}, and under the same assumptions define
\begin{align}\label{til-index}
\wtil{I}(A,f):=\prod_{\frakp\vert p} \wtil{e}_{\frakp}(\one)\cdot I(A,f),
\end{align}
where $\frakp$ runs over the primes of $K$ above $p$ and  $\wtil{e}_{\frakp}(\one)$ is as in Conjecture {$({\textup{BSD}}_{p})$}. 

\begin{conj}\label{bdconj} Consider the setup of \S\ref{sec: theta}.
 Let
 ${\wtil{r}}:=\dim A^{\dagger}(K)_{\Q}$.
\begin{enumerate}
\item Suppose that $\wtil{\eps}=+1$. Then $\wtil{r}$ is even,  $\Theta=\Theta(f)\in \OO(\Y^{-})^{\rm b}$ vanishes at $\chi^{-}=\one$ to order at least $\wtil{r}/2$, the number $\wtil{I}(A,f)$ is the square of an  integer $\wtil{i}(A,f)$,
 and  up to a sign
$$({\rm d}^{-})^{\wtil{r}/2}\Theta(\one)= 
\wtil{i}(A,f) \cdot {\rm pf}^{-}(A)$$
in ${\rm Sym}^{\wtil{r}/2}\Gamma^{-}\otimes \Q$. 
\item Suppose that $\wtil{\eps}=-1$. Then $\wtil{r}$ is odd, ${\mathscr P}={\mathscr P}(f)\in \wtil{H}^{1}_{f}(K,V_{p}A\otimes \OO(\Y^{-})^{\rm b})$
 vanishes at $\chi^{-}=\one$ to order at least $(\wtil{r}-1)/2$,  
the number $\wtil{I}(A,f)$ is the square of an  integer $\wtil{i}(A,f)$,
 and up to a sign
$$({\rm d}^{-})^{(\wtil{r}-1)/2}{\mathscr P}(\one)=
\wtil{i}(A,f) \cdot {\rm Pf}^{-}(A)
 $$
in $A^{\dagger}(K)_{\Q}\otimes {\rm Sym}^{(\wtil{r}-1)/2}\Gamma^{-}$. 
\end{enumerate}
\end{conj}
\begin{rema}\label{takkk-bd} Suppose that $f$ is \emph{optimal}, i.e. either surjective onto $\Z$ (case $\wtil{\eps}=+1$) or with connected kernel (case $\wtil{\eps}=-1$), and moreover  that the following identities hold: (i) $|\Sha(A)|=|\Sha(A)|_{\rm an}$; (ii) $\prod_{v \vert N^{-}} c_{v}(E)= c_{\infty}(E_{K})^{-1}\delta(f)^{-1}$.  
 Then 
 Conjecture \ref{bdconj}  is equivalent   to the original conjecture formulated by Bertolini--Darmon in \cite{bdMT}. Assume that $E$ is isolated in its isogeny class, let $f_{0}\colon J_{0}(N)\to E$ be the optimal  modular parametrisation, and let $c=c_{E}\in\N $ be the Manin constant of $E$ \cite[p. 6]{maz-SD}; it satisfies  $c_{\infty}(E_{K})^{-1}=c^{2}\delta(f_{0})$.
   The identity (ii) is then implied by  the conjecture that $c=1$ together with the main results of Ribet and Takahashi  in \cite{rib-tak, takahashi}, asserting that
 $\delta(f_{0})/\delta(f)= \prod_{v \vert N^{-}} c_{v}(E)$; cf. \cite[Remark 6.6, Theorem 6.7]{bdperiods}.  It is not clear to us whether (ii) should hold in general.
\end{rema}
The conjecture is only made up to sign,  since the right-hand side is insensitive to changing $f$ into $-f$ whereas the left-hand side is not. However we have the following elementary observation about signs. 
\begin{lemm}\label{sign*} Let $*$ be the $\Z_{p}$-algebra involution on $\OO(\Y^{-})^{\rm b}=\Z_{p}\llb \Gamma^{-}\rrb_{\Q_{p}}$ induced by inversion on $\Gamma^{-}$, and  denote the image  under $*$ of an element $s$ of some  $\OO(\Y^{-})^{\rm b}$-module by $s^{*}$. Suppose that $\Theta$ (resp. ${\mathscr P}$) vanishes at $\one$ to order $\geq s$ (resp. $\geq s'$). Then 
$$ ({ \rm d}^{-})^{s}\Theta^{*}(\one) =(-1)^{s}\Theta(\one), \qquad  ({ \rm d}^{-})^{s'}{\mathscr P}^{*}(\one) =(-1)^{s'}{\mathscr P}(\one).$$
\end{lemm}
\begin{proof} The differential of $*$ at $\one\in \Y^{-}$ is ${\rm d}^{-}*(\one)=-1$. 
\end{proof}

\begin{rema} Suppose that $r^{\rm exc}(A)=0$ and $r_{\rm an}(A)\leq 1$. Then Conjecture \ref{bdconj}  follows immediately from the construction of  $\Theta $ and $\mathscr{P}$ and Corollary \ref{index-heeg}.
An elaboration of this observation will result in Proposition \ref{bdvsbsd} below.
\end{rema}

\subsubsection{Evidence} We next  summarise the available evidence towards the Bertolini--Darmon conjecture in cases where $r^{\rm exc}(A)\neq 0$. By  the  Lemma \ref{splitinert}, up to switching $E$ and $E^{(K)}$ we may then suppose that  $E$ has split multiplicative reduction.

\begin{theo}\label{bd evidence}
 Suppose that $E$ has split multiplicative reduction, and let $r_{\rm an}:={\rm ord}_{s=1}L(A,s)$. 
Conjecture \ref{bdconj} holds in the following cases.
\begin{enumerate}
\item\label{bde1} $r_{\rm an}=0$ and $p$ is inert in $K$, as an  exact identity when $E$ is isolated in its isogeny class, and up to a nonzero multiplicative constant otherwise.
\item\label{bde2}  $r_{\rm an}=0$ and $p$ splits in $K$.
\item\label{bde3}  $r_{\rm an}=1$ and $p$ is inert in $K$.
\item\label{bde4}  $r_{\rm an}=1$, $p\geq 5$ splits in $K$, and $N^{-}=1$. 
\end{enumerate}
\end{theo}
\begin{proof} Parts 1--3 are essentially  due to Bertolini--Darmon;
  we shall recall their results,  then deduce Part 4 from  a recent result  of Castella and Molina Blanco.

When $p$ splits in $K$ we write $p\OO_{K}=\frakp\frakp^{*}$. The choice of $\frakp$ or $\frakp^{*}$ will only  affect the formulas which follow up to a sign,  therefore it won't be specified.

 \begin{enumerate}
 \item This is essentially the main result of  \cite{BDrigid}. In its reformulation in \cite[Theorem 5.4, \S5.2]{bdsurvey}, we have 
 $${\mathscr P}(\one)= q_{A,p\OO_{K}} \cdot \wtil{i}(A,f)$$
for some $\wtil{i}(A,f)\in \Z_{p}$, which is denoted  there by $\kappa$,  and equals $c_{p}^{-1}\cdot \lambda_{q_{E}}((y_{n}))$ in the notation of   \cite[Theorem 5.1]{BDrigid}, with $c_{p}=c_{p}(E)={\rm ord}_{p}(q_{E,p})$. (Namely,  $y=\lim_{n} y_{n}\in q_{E,p}^{\times}\hat{\otimes}\Q_{p}$ is identified with $\mathscr{P}(\one)$; and $\lambda_{q_{E}}$, which  also has a geometric interpretation, is simply the homomorphism to $\Q_{p}$  deduced from ${\rm ord}_{p}$.)
It is proved in \emph{loc. cit.} that 
$\lambda_{q_{E}} ((y_{n}))= I(A,f) c_{p}$, hence $$\wtil{i}(A,f)^{2}=c_{p}^{-1}\cdot I(A,f)=\wtil{I}(A,f).$$
That  $\wtil{I}(A,f)$ is in fact a square in $\Z$, under the  additional assumption stated,  follows from the results of Ribet and Takahashi mentioned in Remark \ref{takkk-bd}, together with Corollary \ref{index-heeg}.

\item This is the main result of \cite{bdperiods}. In this case $\wtil{r}=2$ and
$${\rm d}^{-}\Theta(\one) = \mathcal{L}_{\frakp}^{-}(A)\cdot p(f)=  {\rm ord}_{p}(q_{E,p})^{-1 }  \cdot \wtil{h}^{-}(q_{A, \frakp}, q_{A, \frakp})\cdot i(A,f)= \wtil{i}(A,f)\cdot {\rm pf}^{-}(A),$$
where the second equality is by Corollary \ref{index-heeg}.
\item We deduce this from the main result of \cite{bdCD}. In this case $\wtil{r}=2$; we introduce some notation.  Let $X=X_{0}^{N^{-}}(N^{+})$ and  $X':=X_{0}^{N^{-, (p)}}(N^{+}p)$.
 Let $J'$ be the Albanese variety of $X'$. Let $f\colon X\to \Z$ (resp., $f'\colon J'\to E$) be the unique-up-to-sign surjective map  annihilated by the  Hecke operators in $I_{E}$ (resp., the morphism with connected kernel, which  exists after possibly replacing $E$ by an isogenous curve). We may and do assume  that $\Theta=\Theta(f)$ is constructed starting from $f$. Then
$${\rm d}^{-}\Theta (\one)= \log_{A,p\OO_{K}}^{-}(P(f'))= \wtil{h}^{-}(P',
q_{A, p\OO_{K}}) \cdot [A:\Z P(f')] = {\rm pf}^{-}(A)\cdot i(A,f'),$$
where $P'\in A(K)\otimes \Q$ is such that the (generalised) index $[A(K):\Z P']=1$. By Proposition \ref{tak}, we have $\delta(f')=c_{p}^{-1}\delta(f')$; hence with $\wtil{i}(A,f):= i(A, f')$,
 $$\wtil{i}(A,f)^{2}= I(A,f')=c_{p}^{-1}I(A,f)=\wtil{I}(A,f).$$
\item We deduce this from a theorem proved independently by Castella \cite{castella} and Molina Blanco \cite{MB}:\footnote{The results of Molina Blanco hold without the assumption $N^{-}=1$ (and more generally for Shimura curves over totally real fields), but in those additional cases they are formulated in terms of an $\mathcal{L}$-invariant which is not currently known to coincide with the one of our conjectures.} the equality 
\begin{align}\label{cast-f}
{\rm d}^{-}{\mathscr P}(\one)=\mathcal{L}_{\frakp}^{-}(A)\cdot P(f)
\end{align}
holds \emph{in the quotient $H^{1}_{f}(K,V_{p}A)$ of $\wtil{H}^{1}_{f}(K,V_{p}A)$.} (In fact in \cite{castella} the identity  is proved after  applying ${\rm loc}_{\frakp}\colon A(K)\to H^{1}_{f}(K_{\frakp}, V_{p}A)$; under our assumption that ${\rm ord}_{s=1}L(A,1)=1$, by the work of Gross--Zagier and Kolyvagin the $\Q_{p}$-vector space $
H^{1}_{f}(K,V_{p}A)=A(K)_{\Q_{p}}$ is generated by the image of $P$ and the map ${\rm loc}_{\frakp}$ is then injective.) 

We lift \eqref{cast-f} to  the desired identity in $\wtil{H}^{1}_{f}(K, V_{p}A)=A^{\dagger}(K)_{\Q_{p}}$. We abbreviate $P=P(f)$ and assume that the rank of $E$ is $1$ (otherwise, we switch $E$ with its twist $E^{(K)}$).  We make use of two observations. 
\begin{enumerate}
\item[(i)] $\mathscr{P}$ is invariant under the conjugation $c$, hence so is  ${\rm d}^{-}{\mathscr P}(\one)$.
\item[(ii)] Letting  $g\in \Gamma^{-}$ be a generator, the vanishing of ${\mathscr P}(\one )$ implies that ${\mathscr P}$ is divisible by $[g]-1$  in $\Z_{p}\llb \Gamma^{-}\rrb$, and ${\rm d}^{-}{\mathscr P}(\one)$ is precisely the image of $([g]-1)^{-1}\mathscr{P}\in \wtil{H}^{1}_{f}(K, V_{p}A\otimes\Z_{p}\llb \Gamma^{-}\rrb)$ in $\wtil{H}^{1}_{f}(K, V_{p}A)$. As such, it is a `universal norm' for the anticyclotomic extension of $K$ and hence (see \cite[(0.16.1)(iii)]{nek-selmer}) it belongs to the radical of $\wtil{h}^{-}$. 
\end{enumerate}
Noting that $c_{p}:=c_{p}(E)= {\rm ord}_{p}(q_{E, p})=\wtil{e}_{\frakp}(\one)^{-1}=\wtil{e}_{\frakp^{*}}(\one)^{-1}$,  write
$${\rm d}^{-}{\mathscr P}(\one)=c_{p}^{-1}\cdot (a P(f)+b q_{A, \frakp}+ q_{A, \frakp^{*}})$$ 
in terms of the natural basis of $A^{\dagger }(K)$. It follows from \eqref{cast-f} that $a= \ell^{-}(q_{A, \frakp})$, 
 from (i) that $b=c$, and the common value can be fixed by   (ii):
  $$0=c_{p}\cdot\wtil{h}^{-}({\rm d}^{-}{\mathscr P}(\one), q_{A, \frakp})=\ell^{-}(q_{A, \frakp})\log_{A,\frakp}^{-}(P)+ b\cdot \ell^{-}(q_{A, \frakp}), $$
hence the unique (by the non-vanishing \cite{st et} of $\ell^{-}(q)$) solution is $b=-\log_{A,\frakp}^{-}(P)=\log_{A,\frakp^{*}}^{-}(P)$. 

We verify that this implies the desired formula for ${\rm d}^{-}{\mathscr P}(\one)$, according to the definition of ${\rm Pf}^{-}(A)$ in \eqref{pfff}. Let 
$$\gamma=\begin{pmatrix}
\ell^{-}(q_{A, \frakp}) & \log_{A,\frakp^{*}}^{-}(P)& \log_{A,\frakp^{*}}^{-}(P)\\
  & 1&\\
&&1
\end{pmatrix}\in \GL_{3}(\Q_{p}),$$
identified with an element of ${ \GL}(A^{\dagger}(K)_{\Q_{p}})$ by means of the ordered basis $(P, q_{A, \frakp}, q_{A, \frakp^{*}})$.
Let $M':=\gamma A^{\dagger}(K)\oplus A(K)_{\rm tors}$, and let 
$$x':= c_{p}\cdot
{\rm d}^{-}{\mathscr P}(\one)=
 \ell^{-}(q_{A, \frakp}) \cdot P+ \log_{A,\frakp^{*}}^{-}(P)\cdot q_{A, \frakp}+ \log_{A,\frakp^{*}}^{-}(P)\cdot q_{A, \frakp^{*}}\quad \in M'.$$
Then $M'/\Z x'$ is freely generated by $q_{A, \frakp}$ and $q_{A, \frakp^{*}}$, and 
$$[A(K):\Z P]\cdot {\rm Pf}^{-}(A)= \det(\gamma)^{-1}\cdot x' \otimes {\rm pf}\, (M'/\Z x')= \ell^{-}(q_{A, \frakp})^{-1}\cdot x'\otimes \ell^{-}(q_{A, \frakp})=x'.$$
As the last term is  $\wtil{e}_{\frakp}(\one)^{-1}\cdot
{\rm d}^{-}{\mathscr P}(\one) $
by definition, it follows that 
$${\rm d}^{-}{\mathscr P}(\one)= \wtil{e}_{\frakp}(\one)\cdot [A(K):\Z P]\cdot {\rm Pf}^{-}(A), $$
and we conclude by Corollary \ref{index-heeg}.
 \end{enumerate}
\end{proof}

\begin{rema}\label{rem hung}  
The work of Molina Blanco \cite{MB} also  provides a generalisation to totally real fields of the results of \cite{bdperiods}. A  similar generalisation has been independently obtained by Pin-Chi Hung  \cite{hung}.
\end{rema}

\subsection{$p$-adic Gross--Zagier and Waldspurger formulas in anticyclotomic families}\label{sec anti}
Retaining the setup and notation of \S\ref{sec: theta},
we  state formulas relating $\Theta$ and $\mathscr{P} $ to $L_{p}(A)$.

\begin{theo}[Waldspurger formula in anticyclotomic family]\label{wald-a} Suppose that $\wtil{\eps}=+1$.
Then we have
\begin{align}\label{wald-me}
L_{p}(A)|_{\Y^{-}}= {{c}_{\infty}(A)^{-1}} \cdot {\Theta \cdot \Theta^{*} \over \delta(f) } 
\end{align}
in $\Z_{p}\llb \Gamma^{-}\rrb_{\Q_{p}}$. 
\end{theo} 
\begin{proof} This is essentially a special case of \cite[Theorem A]{CH}: more precisely, that result  combined with Theorem B \emph{ibid.} implies that  the following identity (or equivalently all of its specialisations at finite $\chi^{-}\in \Y^{-}$) holds:
\begin{align}\label{wald-CH}
L_{p}(A)|_{\Y^{-}}= \eps(E)\cdot{{c}_{\infty}(A)^{-1}} \cdot {\Theta'\cdot \Theta'^{*}\over \langle f, f\rangle_{R} }.
\end{align}
Here $\Theta'$ is the same as our $\Theta$, except for a different choice of the sequence $g_{n}$; as noted in Remark \ref{comparison}, we then have $\Theta'\cdot \Theta'^{*}= \Theta\cdot \Theta^{*}$. The sign $\eps(E):=\prod_{v}\eps_{v}(E)$ (the product ranging over all places of $\Q$) is the global root number of $E$. Finally, by (3.9) \emph{ibid.}, $\langle f, f\rangle_{R}: =\langle f, \tau f\rangle $ for the pairing \eqref{pairing-B} and an Atkin--Lehner element $\tau=\prod_{v\vert N}\tau_{v}\in \widehat{B}^{\times}$. By Atkin--Lehner theory combined with the Jacquet--Langlands correspondence (see \cite[Theorem 1.2]{bdMT}), $\tau f=\prod_{v\vert N^{+}}\eps_{v}(E)\prod_{v\vert N^{-}}(-\eps_{v}(E))\cdot f$. As $N^{-}$ is the squarefree product of an odd number of factors and $\eps_{\infty}(E)=-1$, we deduce $\tau f= \eps(E) f$. Inserting this into \eqref{wald-CH} yields \eqref{wald-me}.

We recall the steps of  proof of \cite[Theorem A]{CH}, in order to show that the previous argument  can be unwound and then  carried over to the proof of the next theorem. The starting point is the original representation-theoretic Waldspurger formula \cite{wald}  or \cite[(1.4.1)]{yzz}, which in the form used in  \cite[proof of Proposition 3.5]{CH} reads (with slightly different notation)
\begin{align}\label{wald0-CH}
{p_{T}(f_{1}, \chi){ p}_{T}(f_{2}, \baar{\chi}) \over \langle f_{3}, f_{4}\rangle_{\rm Pet}} = 	{\zeta(2)L(1/2, \pi_{K}\otimes \chi)\over 2L(1, \pi, {\rm ad})}\cdot \prod_{v} \int^{*}_{K_{v}^{\times}/\Q_{v}^{\times}} {\langle \pi(t) f_{1,v}, f_{2,v}\rangle_{v}\over \langle f_{3,v}, f_{4,v}\rangle_{v}}\chi(t_{v}) dt_{v},
\end{align}
where $\pi=\otimes_{v}\pi_{v}$ is the automorphic representation with trivial central character of   $B_{\A}^{\times}$ associated with $E$;    $\chi$ is a finite order character of $\A^{\times}\bks {K}_{\A}^{\times}$;  the terms $p_{T}$ are integrals on the torus $\A^{\times}\bks {K}_{\A}^{\times}\subset \A^{\times}\bks B_{\A}^{\times}$;  the pairing $\langle \, , \, \rangle_{\rm Pet}$ (resp. $ \langle \,, \,\rangle_{v}$) is the bilinear Petersson pairing on $\pi\otimes \wtil\pi$ with respect to the Tamagawa measure (resp. an arbitrary pairing on $\pi_{v}\otimes \wtil\pi_{v}$); the $f_{i}=\otimes_{v}f_{i,v}$  are arbitrary elements of $\pi$ (if $i=1,3$) or $\wtil\pi$ (if $i=2,4$) such that $\langle f_{3}, f_{4}\rangle_{\rm Pet}\neq 0$; and finally the notation $\int^{*}$ denotes an integral (for certain local measures $dt_{v}$) normalised  by dividing by  an appropriate product of $L$-values.

Then in \cite{CH}, the formula \eqref{wald-CH} is deduced from \eqref{wald0-CH} by specific choices of the $f_{i}$ which are images of newforms under certain Atkin--Lehner elements (and elements $g_{n}$) and the following steps: (i)  formula for the Asai $L$-value $L(1, \pi, {\rm ad})$; (ii) computation of the Petersson pairing; (iii) computation of local integrals at places $v\nmid p$; (iv)  computation of local integrals at $p$.
For convenience of reference, in the first paragraph of this proof we have deduced \eqref{wald-me} from \eqref{wald0-CH} via \eqref{wald-CH} and some  Atkin--Lehner  functional equations. This is equivalent to directly deducing \eqref{wald-me} from \eqref{wald0-CH} with the following choices: $f_{3}=f_{4}:=f=\baar{f}$ is the $\Z$-valued newform  in $\pi=\baar{\wtil\pi}$; $f_{1}=f_{2}:=f^{(n)}=g_{n}f^{\dagger}$. The local computations corresponding to such choices are the same as   those in \cite[\S3]{CH} except for the addition (or  subtraction) of the ingredient of  Atkin--Lehner functional equations. Steps (i)--(iii)  (for our choices of $f_{i}$) are also in \cite{cst}.
\end{proof}

Let
$$\underline{\wtil{h}}^{ +}\colon \wtil H^{1}_{f}(K, V_{p}A\otimes \OO(\Y^{-})^{\rm b})\otimes \wtil H^{1}_{f}(K, V_{p}A\otimes \OO(\Y^{-})^{\rm b})^{*}\to \OO(\Y^{-})^{\rm b}\otimes \Gamma^{+}$$
 be the `big height pairing' 
    constructed by \nek \cite[\S11]{nek-selmer}. It specialises to the pairing $\wtil{h}^{+}$ at $\chi^{-}=\one\in \Y^{-}$
(and an equally good specialisation property holds for all finite order $\chi^{-}\in \Y^{-}$).

\begin{conj}[$p$-adic Gross--Zagier formula in anticyclotomic family]\label{gza-conj} Suppose  that $E$ is ordinary at $p$ and that $\wtil{\eps}=-1$, so that 
$L_{p}(A)$ vanishes identically along $\Y^{-}$. Let 
${\rm d}^{+}L_{p}(A)|_{\Y^{-}}$ denote the image of $L_{p}(A)$ in $\calI_{\Y^{-}}/\calI_{\Y^{-}}^{2}\cong \Z_{p}\llb \Gamma^{-}\rrb_{\Q_{p}}\otimes_{\Z_{p}} \Gamma^{+}$. 

Then we have
\begin{align}\label{pgza-f}
{\rm d}^{+}L_{p}(A)|_{\Y^{-}}={{c}_{\infty}(A)^{-1}} \cdot {\underline{\wtil{h}}^{ +}(\mathscr{P}, \mathscr{P}^{*}) \over \delta(f) } 
\end{align}
in $\Z_{p}\llb \Gamma^{-}\rrb_{\Q_{p}}\otimes_{\Z_{p}} \Gamma^{+}$. 
\end{conj} 
\begin{theo}\label{gza-th}
 Suppose  that $p$ splits in $K$. Then Conjecture \ref{gza-conj} holds.
\end{theo} 
\begin{proof} This is a consequence of the representation-theoretic version of the same formula from \cite[Theorem C.4]{dd}, in turn based on Theorem B \emph{ibid.}\footnote{When $E$ has good reduction and all $v\vert N$ split in $K$ the result was proved by Howard \cite{howard}.} For a  definition of the relevant representation $\pi$ of $B_{\A}^{\times}$, and a careful discussion of  the various  pairings involved and their relation or analogy to the ones appearing in versions of Waldspurger's formula, we refer to  \cite{cst}.\footnote{The discussion in \cite{cst} is based on \cite{yzz} rather than \cite{dd}, but those two works adopt exactly the same framework. The constants appearing in  the archimedean and $p$-adic Gross--Zagier formula exactly match:  compare \cite[(1.1.3) and Theorem B]{dd}.} 

More precisely,  one can choose a pairing  $\langle\, , \,\rangle$ on $\pi\otimes \wtil\pi$ analogous to the Petersson pairing in  \eqref{wald0-CH} (see \cite[\S2.2]{cst}, local pairings $\langle\, , \,\rangle_{v}$ such that $\langle\, , \,\rangle=\prod_{v}\langle\, , \,\rangle_{v}$, and obtain a version of the $p$-adic Gross--Zagier formula entirely analogous to  \eqref{wald0-CH} by dividing both sides of   \cite[Theorem B]{dd}  by $\langle f_{3}, f_{4}\rangle$. Then 
 the deduction of \eqref{pgza-f} from such formula is obtained  by inserting  the  \emph{same}  choices of $\{f_{i}\}_{i=1,2,3,4}$ and the \emph{same} algebraic computations of steps (i)-(iv)  from the proof of Theorem \ref{wald-a}.\footnote{Step (iv) is also carried out in \cite[Lemma 10.1.2]{dd}.}
\end{proof}

\section{Evidence over imaginary quadratic fields}\label{evidence-imag}

\subsection{Conjecture ({${\text{BSD}}_{p}$}) and the Bertolini--Darmon  conjectures} 
We observe a precise relation between the Bertolini--Darmon conjecture and the $p$-adic Birch and Swinnerton-Dyer conjecture, via the  Waldspurger and $p$-adic  Gross--Zagier  formulas in anticyclotomic families. (In the original work of Bertolini--Darmon \cite{bdMT}, the other terms of comparison  were  not available in general; besides the purely anticyclotomic single-variable  case of  $({\textup{BSD}}_{p})$ which they conjectured, there was rather a relation of 
analogy with the archimedean Birch and Swinnerton-Dyer conjecture, via the archimedean Waldspurger and Gross--Zagier formulas.)
\medskip

We  retain the setup and notation of \S\ref{sec: theta}.  Let $\wtil{r}_{\rm an}:={\rm ord}_{s=1} L(A,s)+|S_{p}^{\rm exc}(A)|$,  $\wtil{r}:=\dim A^{\dagger}(K)_{\Q}$, and 
$$\epsilon:=\begin{cases} 0&\text{if  } \wtil{\eps}=+1\\
1 &\text{if  } \wtil{\eps}=-1.\\
\end{cases}
$$ 

 By the \emph{anticyclotomic Waldspurger/Gross--Zagier formula}, we shall mean either Theorem \ref{wald-a} (case $\wtil{\eps}=+1$) or Conjecture \ref{gza-conj} (case $\wtil{\eps}=-1$).

By the \emph{Bertolini--Darmon Conjecture}, we shall mean Conjecture \ref{bdconj}, \emph{except} for the assertion that the $p$-adic number $\wtil{I}(A,f)^{1/2}$ is an integer.

 By the \emph{almost-anticyclotomic case} of Conjecture {$({\textup{BSD}}_{p})$}, we shall mean the following statement: $({\rm d}^{+})^{\epsilon} L_{p}(A)|_{\Y^{-}}\in \OO(\Y^{-})^{\rm b}\otimes(\Gamma^{+})^{\epsilon}$ vanishes to order $\geq \wtil{r}_{\rm an}-\epsilon$ at $\chi^{-}=\one$, $A^{\dagger}(K)$ has rank $\wtil{r}=\wtil{r}_{\rm an}$, and $$({\rm d}^{-})^{\wtil{r}-\epsilon}({\rm d}^{+})^{\epsilon}L_{p}(A, \one)$$ equals the projection of the right hand side of \eqref{mainid} under 
 $$\pi_{ \epsilon}\colon {\rm Sym}^{\wtil{r}}\Gamma_{K}\to {\rm Sym}^{\wtil{r}-\epsilon}\Gamma^{-}\otimes ({\Gamma}^{+})^{\epsilon}.$$
(Note that, if $\wtil{\eps}=-1$, the restriction $L_{p}(A)|_{\Y^{-}}$ is identically zero by the functional equation; correspondingly $\wtil{R}^{-}(A)=0$ by the property \eqref{h-feq} of $\wtil{h}^{-}$, provided that the rank  of $A^{\dagger}(K)$ is odd as expected.)

 \begin{prop}\label{bdvsbsd} Suppose that Hypothesis $({\rm BSD}_{\infty})$--\ref{BSD1}-\ref{BSD2} and the anticyclotomic  Waldspurger/Gross--Zagier formula  hold for $A$. Then  the Bertolini--Darmon Conjecture for $A$ implies  the almost-anticyclotomic case of Conjecture {$({\textup{BSD}}_{p})$} for $A$. Conversely, the  almost-anticyclotomic case of Conjecture {$({\textup{BSD}}_{p})$} implies the  Bertolini--Darmon Conjecture provided that, in case $\wtil{\eps}=-1$, the term ${\pi}_{1}\left(\wtil{R}(A)\right)\neq 0$. 
 \end{prop}
 \begin{proof} In case $\wtil{\eps}=+1$, note   that the orders of vanishing of $\Theta $ and $\Theta^{*}$ are the same. Then by the anticyclotomic Waldspurger formula, $L_{p}(A)|_{\Y^{-}}$ vanishes to order $\geq \wtil{r}$ if and only if $\Theta$ vanishes to order $\geq \wtil{r}/2$. By Lemma \ref{sign*}, the `leading terms' $({\rm d}^{-})^{\wtil{r}/2}\Theta$ and $({\rm d}^{-})^{\wtil{r}/2}\Theta^{*}$ differ by the sign  $(-1)^{\wtil{r}/2} $. Then
  the desired equivalence is reduced to the identity
$$  \wtil{R}^{-}(A)  = (-1)^{\wtil{r}/2} \cdot {\rm pf}^{-}(A)^{2},$$
which is \eqref{pf2}.

We now turn to the  case $\wtil{\eps}=+1$.
Let $A^{\dagger}(K)_{\Q_{p}}^{\pm}$ be the larger of the eigenspaces under the complex conjugation $c$, and let  $x\in A^{\dagger}(K)_{\Q_{p}}^{\pm}$ be an element in the radical of $\wtil{h}^{-}$. Complete $x=x_{1}$ to an ordered basis ${\mathscr{B}}=\{x_{1}, \ldots, x_{\wtil{r}}\}$ of $A(K)_{\Q_{p}}^{\dagger}$, which we may take to be the image of an integral basis of $A^{\dagger}(K)/A(K)_{\rm tors}$ under an element $\gamma\in {\bf SL}(A^{\dagger}(K)_{\Q_{p}})$.  Denote $M_{1}:=\bigoplus_{i\neq j}^{n}\Z x_{i}\oplus A(K)_{\rm tors}$. We may compute $ {\pi}_{ 1}\left(\wtil{R}(A)\right)$ in the basis ${\mathscr B}$, evaluating  $ {\pi}_{ 1}\left(\det(\wtil{h}(x_{i}, x_{j})_{i, j=1}^{\wtil{r}})\right)$ by expansion along the first row. As  $x=x_{1}$ is in the radical of $\wtil{h}^{-}$, the $j^{\rm th}$ term in the expansion is $  {\pi}_{ 1}\left(\wtil{h}(x,x_{j}){R}_{j}\right)=  \wtil{h}^{+}(x,x_{j}) R_{j}^{-}$, where $R_{j}$ is $(-1)^{1+j}$ times the $(1,j)^{\rm th}$ minor and $R_{j}^{-}=\pi_{0}(R_{j})$.   If $j\neq 1$ this minor in fact vanishes as the first column of the corresponding matrix is zero. We conclude that 
 $$  {\pi}_{ 1}\left(\wtil{R}(A)\right)=   \wtil{h}^{+}(x,x) R(M_{2}, \wtil{h}^{-})= (-1)^{(\wtil{r}-1)/2}\cdot \wtil{h}^{+}\left({\rm Pf}^{-}(A),{ \rm Pf}^{-}(A)\right),$$
where the last equality uses  \eqref{pf2} applied to $M_{2}$. 
By this identity, the anticyclotomic Gross--Zagier formula, and again Lemma \ref{sign*}, we see that the Bertolini--Darmon Conjecture implies the almost-anticyclotomic case of Conjecture {$({\text{BSD}}_{p})$}. The implication can be reversed if $\wtil{h}^{+}$ is non-vanishing on ${\rm Pf}^{-}(A)$.
\end{proof}
 
\subsection{Proof of Theorems \ref{mainK} and \ref{exc}}\label{sec: bc}
We proceed to prove our main theorems, after  some preliminaries which allow to settle `trivial' cases.
 \begin{prop}\label{feq}
 \begin{enumerate}
 \item\label{p1}
We have $$\wtil \eps=(-1)^{\wtil{r}_{\rm an}}.$$
 \item\label{p2} Suppose that $L_{p}(A)$ vanishes to order  $\geq i+j$ at $\chi=\one$. Then if $i\not\equiv \wtil{r}\pmod 2$, we have
$$({\rm d}^{+})^{i}
({\rm d}^{-})^{j}
L_{p}(A, \one)=0$$
in $({\rm Sym}^{i}\Gamma^{+}\otimes {\rm Sym}^{j}\Gamma^{-})\otimes L$.
\item\label{p3} Suppose that  $L_{p}(A)$ vanishes to order  $\geq\wtil{r}_{\rm an}$ at $\chi=\one$. Then 
 $${\rm d}^{\wtil{r}_{\rm an}}L_{p}(A, \one) \in ({\rm Sym}^{\wtil{r}}\Gamma)_{L}^{+},$$
the   $+1$-eigenspace for the action of the complex conjugation $c\in \Gal(K/\Q)$ on $({\rm Sym}^{\wtil{r}}\Gamma)\otimes L$; that is, $({\rm d}^{+})^{\wtil{r}_{\rm an}-j}({\rm d}^{-})^{j}L_{p}(A, \one)=0$ in $({\rm Sym}^{\wtil{r}_{\rm an}-j}\Gamma^{+}\otimes {\rm Sym}^{j}\Gamma^{-})\otimes L$ for every odd $j$.
 \end{enumerate}
 \end{prop}
\begin{proof} The sign of the functional equation for $L(A,s)$ is $-\eta(N)=(-1)^{r_{\rm an}}$; then part \ref{p1} is equivalent to the statement that $r^{\rm exc}$ is odd if and only if $p\vert N$ and $p$ is inert in $K$, which follows from Lemma \ref{splitinert}. Part \ref{p2} follows from part \ref{p1} and the functional equation for $L_{p}(A)$, and part \ref{p3} from part \ref{p2}.
\end{proof}

 The following is the  counterpart of the previous proposition on the arithmetic side.
 \begin{prop}\label{feq2} Let $\wtil r$ be the rank of $A^{\dag}(K)_{\Q_{p}}$.
 Then $$\wtil{R}(A) \in ({\rm Sym}^{\wtil{r}}\Gamma)_{L}^{+};$$ 
that is, the  image of $\wtil{R}(A)$ in $({\rm Sym}^{\wtil{r}-j}\Gamma^{+}\otimes {\rm Sym}^{j}\Gamma^{-})\otimes L$ vanishes for every odd $j$. 
 \end{prop}
 \begin{proof}
 By Lemma \ref{equivt}, the map 
 $$\det{\wtil{h}}\colon \det A^{\dag}(K)_{\Q_{p}} \otimes \det A^{\dag}(K)_{\Q_{p}} \to ({\rm Sym}^{\wtil{r}}\Gamma)\otimes L$$
 is $c$-equivariant.  As $c^{2}=1$, $c$ acts by $\pm 1$ on $ \det A^{\dag}(K)_{\Q_{p}}$ and by $(\pm 1)^{2}=+1$ on  $ \det A^{\dag}(K)_{\Q_{p}}^{\otimes 2} $. The result follows.
 \end{proof}

The proof of Theorem \ref{mainK} for is naturally subdivided into several cases corresponding to the projections  $\pi_{i}\colon {\rm Sym}^{\wtil{r}}\Gamma 
\to
({\rm Sym}^{i}\Gamma^{+}\otimes {\rm Sym}^{\wtil{r}-i}\Gamma^{-})$.
 The identity between the images of both sides of \eqref{mainid}  under the above projection will be referred to 
 as the \emph{purely cyclotomic case} if $i=\wtil{r}$,  as the \emph{purely anticyclotomic case} if $i=0$, and as a \emph{mixed case} if $0<i<\wtil{r}$.

By Proposition \ref{invar}, the purely cyclotomic case is always reduced to the case $K=\Q$ treated before.

\subsubsection{Case $r=0$, $p$ inert}  This is trivial unless $A$ has (necessarily split, by Lemma \ref{splitinert}) multiplicative reduction, in which case we have $\wtil{r}=1$ for $A$. The cyclotomic case is reduced to Theorem \ref{Q} as remarked before.  By Propositions \ref{feq}.\ref{p3} and \ref{feq2}, the anticyclotomic case is reduced to the equality $0=0$. 
\begin{rema} In this context, the purely cyclotomic and the almost-anticyclotomic case of   {$({\text{BSD}}_{p})$} coincide. A  proof using Proposition \ref{bdvsbsd}  would presently  only be conditional, as  the  $p$-adic Gross--Zagier formula  is not available when $p$ is  inert $K$. Nevertheless it seems worth highlighting that such a formula, combined with Bertolini--Darmon's   $$\mathscr{P}(\one)\doteq q_{A, p\OO_{K}}$$ (Theorem \ref{bd evidence}.\ref{bde1}) would yield an intriguing  (and difficult) proof of the result of Greenberg--Stevens.
 \end{rema}

\subsubsection{Case $r=0$, $p$ split.} Again the only non-trivial case is the one of split multiplicative reduction with $\wtil{r}=2$ and $A^{\dagger}(K)_{\Q_{p}}$ generated by $q_{A, \frakp}\in K_{\frakp}^{\times}$ and $q_{A, \frakp^{*}}\in K_{\frakp^{*}}^{\times}$ for $S_{p}=\{\frakp, \frakp^{*}\}$. First we show  the assertion on the order of vanishing, i.e. that ${\rm d}\,L_{p}(A, \one)=0$: the component ${\rm d}^{+}L_{p}(A, \one)=0$ by 
the factorisation 
\begin{align}\label{factor}
L_{p}(A)|_{\Y^{+}}={\Omega_{E}^{+}\Omega_{E}^{-}\over |D_{K}|^{-1/2}\Omega_{A}}   L_{p}(E)L_{p}(E^{(K)}),\end{align}
 and the component ${\rm d}^{-}L_{p}(A ,\one)=0$ by 
 the anticyclotomic Waldspurger formula (Theorem \ref{wald-a}), since both $\Theta$, $\Theta^{*}\in \Z_{p}\llb \Gamma^{-}\rrb$ vanish at $\chi^{-}=\one$.

We can now deal prove  the identity \eqref{mainid}.
The  purely cyclotomic case is as usual reduced to Theorem \ref{Q}. 
The mixed case is again trivial, i.e. reduced to $0=0$, by Propositions \ref{feq} and \ref{feq2}.
The purely anticyclotomic case follows from Proposition \ref{bdvsbsd} and Theorem \ref{bd evidence}.\ref{bde2}:  explicitly, it is the identity
$$({\rm d}^{-})^{2}L_{p}(A, \one)=\mathcal{L}_{\frakp}^{-}(A) \mathcal{L}_{\frakp^{*}}^{-}(A) \cdot L_{\rm alg}(A, 1)\qquad \text{in } (\Gamma^{-})^{\otimes 2}_{\Q_{p}}.$$

\subsubsection{Case $r=1$, $p$ inert} The purely cyclotomic case is reduced to Theorem \ref{Q}. If $A$ has good reduction at $p$, then $\wtil{r}=r=1$ and the purely anticyclotomic case is trivial by Propositions \ref{feq}, \ref{feq2}.
Suppose that $A$ has (necessarily split) multiplicative reduction. Then $\wtil{r}=2$.  For the order of vanishing, we have ${\rm d}^{+}L_{p}(A, \one)=0$ by \eqref{factor},
and Theorem  \ref{wald-a} again implies
 ${\rm d}^{-}L_{p}(A ,\one)=0$. 

The mixed case is again trivial. 
Finally, the purely anticyclotomic case follows from  Proposition \ref{bdvsbsd} and Theorem \ref{bd evidence}.\ref{bde3}: explicitly, it is the identity
$$ ({\rm d}^{-})^{2} L_{p}(A, \one)= -{{\log_{A, p\OO_{K}}^{-}(P)}^{2}\over {\rm ord}_{p}(q_{A,p})}  \cdot L'_{\rm alg}(A,1)$$
in $(\Gamma^{-})^{\otimes 2}\otimes\Q_{p}$, where $P\in A(K)_{\Q}$ is such that the generalised index $[A(K):\Z P]=1$.

\subsubsection{Case $r=1$, $p$ split} The purely cyclotomic case is reduced to Theorem \ref{Q}. If $E$ does not have split multiplicative reduction, then $\wtil{r}(A)=r(A)=1$, and the anticyclotomic case is trivial by Propositions \ref{feq}, \ref{feq2}. If $E$ has split multiplicative reduction, then $\wtil{r}(A)=r+2=3$. We first prove the assertion on the order of vanishing. By the functional equation we have $({\rm d}^{+})^{i}L_{p}(A)|_{\Y^{-}}=0$ for every even $i$. Then we only need to show $({\rm d}^{-})^{j}{\rm d}^{+}L_{p}(A, \one)=0$ for $j=0,1$. For $j=0$ this follows from the factorisation \eqref{factor}  and Theorem \ref{Q}. Consider the case $j=1$. By the  $p$-adic Gross--Zagier formula in anticyclotomic families (Theorem  \ref{gza-th}),
${\rm d}^{+}L_{p}(A)|_{\Y^{-}}$ is  a multiple of the image of
 $\mathscr{P}\otimes \mathscr{P}^{*}$ under $\underline{\wtil{h}}^{+}$ on $\Y^{-}$. As  both $\mathscr{P}$, $\mathscr{P}^{*}$  have a trivial zero at $\chi^{-}=\one$. We conclude that ${\rm d}^{-}{\rm d}^{+}L_{p}(A, \one)=0$ and so 
 $L_{p}(A)$ vanishes to order at least $\wtil{r}=3$ at $\chi=\one$.
 
We now turn  to the main identity \eqref{mainid}. By Propositions \ref{feq}, \ref{feq2}, the only nontrivial cases are  the purely cyclotomic and   the almost-anticyclotomic ones. The latter follows from Proposition \ref{bdvsbsd} and Theorems \ref{gza-th} and \ref{bd evidence}.\ref{bde4}.

The almost-anticyclotomic  case is equivalent to the
 following, which is precisely Theorem \ref{exc} in the Introduction. 

\begin{theo}\label{exc'} Let $E/\Q$ be an elliptic curve with conductor $N$ and  split multiplicative reduction at the prime $p\geq 5$. Let $K$ be an imaginary quadratic field such that all  primes dividing $N$ split in $K$ and let $A:=E_{K}$; then $S_{p}^{\rm exc}(A)=S_{p}=\{\frakp, \frakp^{*}\}$.
Suppose that ${\rm ord}_{s=1}L(A,s)=1$. We have  ${\rm d}^{i}L_{p}(A)=0$  for all $i\leq 2$, and 
\begin{gather}\label{id exc} 
({\rm d}^{-})^{2}{\rm d}^{+}L_{p}(A, \one)=\mathcal{L}_{\frakp}^{-}(A)\mathcal{L}_{\frakp^{*}}^{-}(A)\cdot {R}^{\rm norm, +}(A)\cdot 
{L'_{\rm alg}(A, 1)}
\end{gather}	
in $({\rm Sym}^{2}\Gamma^{-}\otimes \Gamma^{+})_{\Q_{p}}$.
\end{theo}

\begin{proof} We give two proofs. 
\begin{enumerate}
\item Observe first that $R^{\rm norm,-}(A)=0$ as $A(K)_{\Q_{p}}$ is a $c$-eigenspace and $h^{\rm norm,-}$ satisfies \eqref{h-feq}. Then by Proposition \ref{compatibility}, the right-hand side of \eqref{id exc} equals the image under $\pi_{1}$ of the right-hand side of \eqref{mainid}. Since under our assumption we have just proved that \eqref{mainid} holds, so does \eqref{id exc}. 
\item We can unwind the previous argument to give a more self-contained (but not essentially different) proof. Let $\baar{\mathscr P}$ be $\mathscr P$ viewed as an element   of  $H^{1}_{f}(K,V_{p}A, \OO(\Y^{-})^{\rm b})$ rather than $\wtil H^{1}_{f}(K,V_{p}A, \OO(\Y^{-})^{\rm b})$, and consider the pairing 
$$\underline{h}^{\rm norm, +}\colon H^{1}_{f}(K,V_{p}A, \OO(\Y^{-})^{\rm b})\otimes H^{1}_{f}(K,V_{p}A, \OO(\Y^{-})^{\rm b})^{*}\to  \OO(\Y^{-})^{\rm b}\otimes \Gamma^{+}$$
constructed in \cite{PR2}: its specialisations coincide with those of $\wtil{h}^{+}$ at all finite order $\chi^{-}\in\Y^{-}-\{\one\}$, whereas at  $\chi^{-}=\one$ it specialises to ${h}^{\rm norm, +} $. We may then rewrite \eqref{pgza-f} in the form 
$${\rm d}^{+}L_{p}(A)|_{\Y^{-}}={{c}_{\infty}(A)^{-1}} \cdot {\underline{{h}}^{\rm norm, +}\left(\baar{\mathscr{P}}, \baar{\mathscr{P}}^{*}\right) \over \delta(f) } $$
(since the specialisations of both formulas at all finite order $\chi^{-}\in \Y^{-} $ coincide: at $\chi^{-}=\one $ both reduce to $0=0$). We can then  directly obtain Theorem \ref{exc'} by applying $({\rm d}^{-})^{2}$, inserting the formula 
${\rm d}^{-}{\mathscr P}(\one)=\mathcal{L}_{\frakp}^{-}(A)\cdot P(f)$ of \eqref{cast-f}, and fixing constants via Corollary \ref{index-heeg}.
\end{enumerate}
\end{proof}

\backmatter
\addtocontents{toc}{\medskip}

\end{document}